\documentclass[a4paper]{amsart}
\usepackage[T1]{fontenc}
\usepackage[latin1]{inputenc}
\usepackage[english]{babel}
\usepackage{amsmath,amsthm,latexsym,amssymb,mathrsfs}
\DeclareMathAlphabet{\altmathcal}{OMS}{cmsy}{m}{n}
\usepackage{mathptmx}
\usepackage{dsfont} 
\usepackage{mathtools}
\usepackage{color}
\usepackage{physics}
\usepackage{faktor}
\usepackage[shortlabels]{enumitem}
\usepackage{hyperref}
\usepackage{graphicx}
\usepackage{tabularx}
\newcolumntype{C}{>{\centering\arraybackslash}X} 
\usepackage{csquotes}
\usepackage{url}
\usepackage{import}
\usepackage[pass]{geometry}

\usepackage[style=alphabetic,natbib=true,url=false,doi=false,isbn=false]{biblatex}
\renewbibmacro{in:}{}
\newtheorem{theorem}{Theorem}[section]
\newtheorem*{theorem*}{Theorem}
\newtheorem{theoremx}{Theorem}
\newtheorem{lemma}[theorem]{Lemma}

\newtheorem{proposition}[theorem]{Proposition}
\theoremstyle{definition}
\newtheorem{definition}[theorem]{Definition}

\theoremstyle{remark}
\newtheorem{remark}[theorem]{Remark}

\newcommand{\defin}{\vcentcolon =}
\newcommand{\C}{\mathbb{C}}
\newcommand{\R}{\mathbb{R}}
\newcommand{\N}{\mathbb{N}}
\newcommand{\Z}{\mathbb{Z}}
\newcommand{\I}{I}
\newcommand{\II}{I  \! \! I}

\newcommand{\Hyp}{\mathbb{H}}
\newcommand{\dS}{\mathrm{dS}}
\newcommand{\Teich}{\altmathcal{T}}

\newcommand{\length}{\ell}
\newcommand{\CC}{C}
\newcommand{\ad}{\textit{ad}}
\newcommand{\neigh}{N}
\newcommand{\surf}{S}

\newcommand{\id}{\textit{id}}

\newcommand{\1}{\mathds{1}}

\newcommand{\GeoLam}{\altmathcal{GL}}
\newcommand{\MesLam}{\altmathcal{ML}}

\newcommand{\mappa}[3]{#1 \colon #2 \rightarrow #3}
\newcommand{\hsk}{\hskip0pt}

\DeclarePairedDelimiterX{\scal}[2]{\langle}{\rangle}{#1, #2}
\DeclarePairedDelimiterX{\scall}[2]{(}{)}{#1, #2}
\DeclarePairedDelimiter{\set}{\{}{\}}

\DeclareMathOperator{\supp}{supp}

\DeclareMathOperator{\Vol}{Vol}
\DeclareMathOperator{\Area}{Area}

\DeclareMathOperator{\Iso}{Iso}
\DeclareMathOperator{\hol}{hol}
\DeclareMathOperator{\Proj}{P}

\DeclareMathOperator{\ext}{ext}

\DeclareMathOperator{\interior}{int}
\DeclareMathOperator{\dvol}{dvol}

\title{The dual Bonahon-Schl{\"a}fli formula}

\author{Filippo Mazzoli}

\date{\today}

\address{Department of Mathematics, 
	University of Virginia,
	Charlottesville, VA, 
	United States}

\email{filippomazzoli@me.com}

\subjclass[2010]{Primary: 53C65, 57M50; Secondary: 30F40, 52A15, 57N10}

\begin{document}
    
    \begin{abstract}    
        	Given a differentiable deformation of geometrically finite hyperbolic $3$-\hsk manifolds $(M_t)_t$, the \emph{Bonahon-Schl\"afli formula} \cite{bonahon1998schlafli} expresses the derivative of the volume of the convex cores $(\CC M_t)_t$ in terms of the variation of the geometry of its boundary, as the classical \emph{Schl{\"a}fli formula} \cite{schlafli1858multiple} does for the volume of hyperbolic polyhedra. Here we study the analogous problem for the \emph{dual volume}, a notion that arises from the polarity relation between the hyperbolic space $\Hyp^3$ and the de Sitter space $\dS^3$. The corresponding \emph{dual Bonahon-Schl\"afli formula} has been originally deduced from Bonahon's work by Krasnov and Schlenker \cite{krasnov2009symplectic}. Applying the \emph{differential Schl\"afli formula} \cite{schlenker_rivin1999schlafli} and the properties of the dual volume, we give a (almost) self-contained proof of the dual Bonahon-Schl\"afli formula, without making use of the results in \cite{bonahon1998schlafli}.
    \end{abstract}
    
    \maketitle
    
    
    \section*{Introduction}
    
    The classical Schl{\"a}fli formula expresses the derivative of the volume along a $1$-\hsk parameter deformation of polyhedra in terms of the variation of its boundary geometry. It was originally proved by Schl\"afli \cite{schlafli1858multiple} in the unit $3$-\hsk sphere case, and later extended to polyhedra of any dimension sitting inside constant non-zero sectional curvature space forms of any dimension. Here we recall the statement in the $3$-\hsk dimensional hyperbolic space $\Hyp^3$, which will be our case of interest:
    
    \begin{theorem*}[Schl\"afli formula] Let $(P_t)_t$ be a $1$-\hsk parameter family of polyhedra in $\Hyp^3$ having the same combinatorics, obtained by taking a differentiable variation of the vertices of $P = P_0$. Then the function $t \mapsto V_t = \Vol(P_t)$ is differentiable at $t = 0$ and it verifies
        \[
        \dot{V} = \frac{1}{2} \sum_{\substack{\text{$e$ edge} \\ \text{of $\partial P$}}} \length(e) \; \dot{\theta} (e) ,
        \]
        where $\length(e)$ denotes the length of the edge $e$ in $P$ and $\dot{\theta}(e)$ is the variation in $t$ of the exterior dihedral angle along $e$.
    \end{theorem*}
    
    Bonahon \cite{bonahon1998schlafli} proved an analogue of this result for variations of hyperbolic $3$-\hsk manifolds. More precisely, consider a differentiable $1$-\hsk parameter family of quasi-isometric geometrically finite hyperbolic $3$-manifolds $(M_t)_t$; in any of such $M_t$'s there is a smallest convex subset $\CC M_t$, called the \emph{convex core} of $M_t$, which plays the role of the polyhedron. It has a boundary $\partial \CC M_t$ which is totally geodesic almost everywhere, except for a closed subset $\lambda_t$ foliated by simple geodesics, where the surface $\partial \CC M_t$ is bent. The structure of $\partial \CC M_t$ is encoded in the datum of a hyperbolic metric $m_t$, obtained by gluing the metrics on the complementary regions of $\lambda_t$, and a \emph{measured lamination} $\mu_t$, which describes the amount of bending of $\partial \CC M_t$ along $\lambda_t$. The \emph{geodesic lamination} $\lambda_t$ is the analogue of the $1$-skeleton in the boundary of the polyhedron, and the bending measure $\mu_t$ is the integral sum of the dihedral angles along the transverse arcs to $\lambda_t$.
    
    The space of measured laminations $\MesLam(\partial \CC M) = \MesLam(\partial \CC M_t)$ is naturally endowed with a piecewise linear manifold structure, therefore the tangent directions at the point $\mu_0$ form in general a union of cones, each of which is sitting in the tangent space of some linear piece. Bonahon's notion of \emph{H\"older cocycles} (see \cite{bonahon1997geodesic}, \cite{bonahon1997transverse}) furnishes a natural way to describe these first-\hsk order variations of measured laminations. In \cite{bonahon1998variations} the study of the dependence of $m_t$ and $\mu_t$ in terms of the hyperbolic structure $M_t$ is developed. In particular, the hyperbolic metric $m_t$ is shown to depend $\mathscr{C}^1$ in the parameter, and the measured lamination always admits left and right derivatives in $t$, which is the best that can be expected in a piecewise linear setting. In light of these facts, Bonahon showed in \cite{bonahon1998schlafli} that, for a $1$-\hsk parameter family of manifolds $(M_t)_t$ as above, the volume of the convex core $V_\CC(M_t) \defin \Vol(\CC M_t)$ always admits right (and left) derivative at $t = 0$, and verifies
    \[
    \left. \dv{V_\CC(M_t)}{t} \right|_{t = 0^+} = \frac{1}{2} \length_m (\dot{\mu}_{0^+}) ,
    \]
    where $m = m_0$ and $\dot{\mu}_{0^+} = \dv{\mu_t}{t} |_{t = 0^+}$. We will call this relation the \emph{Bonahon-Schl{\"a}fli formula}.
    
    Another notion of volume can be introduced on the space of convex subsets sitting inside a convex co-compact hyperbolic manifold $M$ (for simplicity, here we require $\CC M_t$ to be not only of finite volume but also compact). Namely, we can define the \emph{dual volume} of a compact convex subset $N$ of $M$ with smooth boundary by the following relation:
    \begin{equation} \label{eq:dual_volume}
        \Vol^*(N) \defin \Vol(N) - \frac{1}{2} \int_{\partial N} H \dd{a},
    \end{equation}
    where $H$ denotes the trace of the shape operator of $\partial N$, defined by its interior unitary normal vector field. This notion is related to the duality between the hyperbolic space $\Hyp^3$ and the de Sitter space $\dS^3$ (see for instance \cite{hodgson_rivin1993a_characterization}), which allows one to associate with a convex body $C$ in one geometry, a dual one $C^\wedge$ sitting in the other. By applying the definition (\ref{eq:dual_volume}) to a compact convex body $C \subset \Hyp^3$, $- \Vol^*(C)$ turns out to be the de Sitter volume of $H \setminus C^\wedge$, where $H$ is a future-oriented half-space containing $C^\wedge$.
    
    Krasnov and Schlenker \cite{krasnov2009symplectic} deduced a variation formula for the dual volume of the convex cores $(\CC M_t)_t$ from the Bonahon-Schl\"afli formula. More precisely, they showed:
    \begin{theoremx} \label{thmx:dual_bonahon_schlafli}
        The derivative at $t = 0$ of $V_\CC^*(M_t) \defin \Vol^*(\CC M_t)$ exists and it verifies 
        \[
        \dd{V_\CC^*}(\dot{M}) = - \frac{1}{2} \dd{L_\mu} (\dot{m}),
        \]
        where $L_\mu$ denotes the function on the Teichm\"uller space of $\partial \CC M$, which associates to a hyperbolic metric $m \in \Teich(\partial \CC M)$ the length of $\mu = \mu_0$ with respect to $m$, and $\dot{m}$ denotes the variation of the hyperbolic metric on the boundary of the convex core at $t = 0$.
    \end{theoremx}
    
    The remarkable property of this relation, which we call the \emph{dual Bonahon-Schl{\"a}fli formula}, is that it does not involve the first variation of the bending measures $\dot{\mu}_{0^+}$, but only the derivative of the hyperbolic metric $m_t$. Therefore, contrary to the variation formula of the "standard" hyperbolic volume, this relation does not require the notion of H\"older cocycle to be stated. A fairly natural question (suggested in \cite{krasnov2009symplectic}) is to understand whether it is possible to find a proof of Theorem \ref{thmx:dual_bonahon_schlafli} that does not involve the study of the variation of the bending measures of the convex core, which could possibly simplify the proof of the statement. The purpose of this paper is to give an affirmative answer to this question.
    
    Even if inspired by Bonahon's work, our strategy of proof is quite different from the one used in \cite{bonahon1998schlafli} and mainly relies on tools from differential geometry, including the \emph{differential Schl\"afli formula} \cite{schlenker_rivin1999schlafli} and the convexity properties of the equidistant surfaces from the convex core. Without making use of the H\"older cocycles technology, we will prove that the derivative at $t = 0$ of the dual volume of the convex core exists and it verifies
    \[
    \dd{V_\CC^*}(\dot{M}) = - \frac{1}{2} \left. \dv{\length_{M_t}(\mu)}{t} \right|_{t = 0} ,
    \]
    where $\length_{M_t}(\mu)$ is the length of the measured lamination $\mu = \mu_0$ \emph{realized inside the manifold $M_t$}, as $t$ varies in a neighborhood of $0$. In order to deduce that the term $\dv{t} \length_{M_t}(\mu) |_{t = 0}$ coincides with $\dd{L_\mu}(\dot{m})$, and therefore the complete statement, we will need Bonahon's results about the $\mathscr{C}^1$-\hsk dependence of the hyperbolic metric on the boundary of the convex core with respect to the convex co-compact structure of $M$ (see \cite[Theorem~1]{bonahon1998variations}).
    
    Our interest in the variation formula of the dual volume is also motivated by its consequences in the study of the geometry of quasi-Fuchsian manifolds. In an upcoming work \cite{mazzoli2019dual_volume_WP}, the author has shown how the dual Bonahon-Schl\"afli formula can be used to produce an explicit linear bound of the dual volume of the convex core of a quasi-Fuchsian manifold $M$ in terms of the Weil-Petersson distance between the two hyperbolic metrics on the boundary of the convex core of $M$, in analogy to what has been done by Schlenker \cite{schlenker2013renormalized} using the notion of \emph{renormalized volume} (see \cite{krasnov2008renormalized}).
    
    Concerning the renormalized volume, the dual Bonahon-Schl\"afli formula is the counterpart "at the convex core" of another remarkable relation, which was proved in \cite{schlenker2017notes} and concerns the geometry "at infinity" of convex co-compact hyperbolic manifolds and their renormalized volume $V_R$. More precisely, let $(M_t)_t$ be a $1$-\hsk parameter family of quasi-Fuchsian manifolds, with conformal classes at infinity given by $c_t \in \Teich (\partial_\infty M)$. The boundary at infinity of $M_t$ is naturally endowed with a complex projective structure $\sigma_t$, with underlying conformal structure $c_t$. We denote by $\altmathcal{F}_t$ the horizontal measured foliation of the Schwarzian quadratic differential associated to the structure $\sigma_t$ with respect to the uniformized hyperbolic structure of $c_t$ (see \cite{schlenker2017notes} for details). Then, the derivative at $t = 0$ of the renormalized volume $V_R$ of $M_t$ can be expressed as
    \[
    \dd{V_R} (\dot{M}) = - \frac{1}{2} \dd{\ext_\altmathcal{F}}(\dot{c}) ,
    \]
    where $\ext_\altmathcal{F}$ is the \emph{extremal length} of $\altmathcal{F} = \altmathcal{F}_0$, considered as a function over $\Teich(\partial_\infty M)$ (here the Teichm\"uller space is thought as space of Riemann surface structures over $\partial_\infty M$). As described in \cite{schlenker2017notes}, this is one of several interesting results where the quantities $m_t$, $\mu_t$ and $V_\CC^*$, at the boundary of the convex core $\partial \CC M_t$, relate to each other as $c_t$, $\altmathcal{F}_t$ and $V_R$ do at the boundary at infinity $\partial_\infty M$.
    
    \subsection*{Outline of the paper}
    In Section $1$ we recall the notions of convex co-compact hyperbolic $3$-\hsk manifold and of equidistant surfaces from the convex core $\CC M_t$, on which we will base large part of our analysis, and we describe a procedure to locally approximate the boundary of the convex core $\partial \CC M$ by finitely bent pleated surfaces. Section $2$ is dedicated to the notion of dual volume and the description its properties. In Section $3$ we describe a formula for the derivative of the length of a measured lamination realized in a hyperbolic manifold $M$, which will be used to express the term $\dv{t} \length_{M_t}(\mu) |_{t = 0}$.
    
    Section $4$ is the central part of our proof. Firstly we will approximate the convex cores $\CC M_t$ by their $\varepsilon$-\hsk neighborhoods $N_\varepsilon \CC M_t$. Fixing the underlying topological space and varying the hyperbolic structures $M_t$ regularly enough, we will study for which values of $\varepsilon$ and $t$ the surfaces $N_\varepsilon \CC M_0$ remain convex with respect to the structure of $M_t$. This will allow us to estimate the dual volumes of the convex cores $\CC M_t$ with the dual volumes of the regions $N_\varepsilon \CC M_0$. Here the key properties that will play a role are the minimality of the convex core among all convex subsets, and the monotonicity of the dual volume with respect to the inclusion. In this way we will be able to deduce the variation of the dual volume of the convex core from the one of a more regular family of convex regions, on which in particular we are able to apply a "smooth analogue" of the classical Schl\"afli formula, proved in \cite{schlenker_rivin1999schlafli}. At this level it is possible to see a major difference between the variation of the volume and the one of the dual volume, which is that the latter one involves only the derivative of the Riemannian metric restricted to the surface (see Proposition \ref{prop:variation_formula_dual_volume_in_M}), while in the first one appears also the variation of the mean curvature (compare with \cite{schlenker_rivin1999schlafli}). This characteristic explains why the dual volume turns out to be easier to handle than the standard Riemannian volume. Finally, in the end of the paper we will deduce Theorem \ref{thmx:dual_bonahon_schlafli} by combining these observations and using an approximation argument.
    
    \subsection*{Acknowledgments}
    I would like to thank my advisor Jean-Marc Schlenker for his help and support all along this work, and Andrea Seppi for his advise to improve the exposition. I also would like to thank the referee for careful reading of the manuscript, with many corrections and useful comments. This work has been supported by the Luxembourg National Research Fund PRIDE15/10949314/GSM/Wiese.
    
    \section{Convex co-compact manifolds} \label{section:equidistant_surfaces}
    
    In this section, we recall the main geometric objects involved in our study. In particular we will recall the definition of convex co-compact hyperbolic $3$-\hsk manifold and the structure of its convex core, which is described by an hyperbolic metric and a measured lamination. Later we will state some geometric properties of equidistant surfaces from planes and lines in $\Hyp^3$, and finally we will give a procedure to approximate the lift of the convex core to the universal cover by finitely bent surfaces. These will be useful technical ingredients for the rest of our exposition.
    
    \vspace{0.5cm}
    
    Let $M$ be a complete hyperbolic $3$-\hsk manifold, namely a $3$-\hsk manifold endowed with a complete Riemannian metric having sectional curvature constantly equal to $-1$. A subset $C \subseteq M$ is \emph{convex} if for any choice of distinct points and for every geodesic arc $\gamma$ in $M$ connecting them, $\gamma$ is fully contained in $C$. Then $M$ is said to be \emph{convex co-compact} if $M$ has a non-empty compact convex subset $C$. It turns out that, if $M$ is a convex co-compact hyperbolic manifold, there exists a smallest compact convex subset with respect to the inclusion, called the \emph{convex core} of $M$ and denoted by $\CC M$.
    
    The boundary of the convex core is the union of a finite collection of connected surfaces, each of which is totally geodesic outside a subset having Hausdorff dimension $1$. As described in \cite{canary_epstein_green2006}, the hyperbolic metrics on the flat parts "merge" together, defining a complete hyperbolic metric $m$ on $\partial \CC M$. The locus where $\partial \CC M$ is not flat is a \emph{geodesic lamination} $\lambda$, namely a closed subset of $\partial \CC M$ which is union of disjoint simple $m$-\hsk geodesics, called the \emph{leaves} of the lamination. The surface $\partial \CC M$ is bent along $\lambda$, and the amount of bending can be described by a measured lamination. More precisely, a \emph{measured lamination} $\mu$ is a collection of regular positive measures, one for each arc transverse to a lamination $\lambda$, verifying two natural compatibility conditions: if $c$ is a transverse arc and $c'$ is a subarc of $c$, then the measure associated to $c'$ is the restriction to $c'$ of the measure of $c$; the measures are invariant under isotopies between transverse arcs. In particular, the \emph{bending measure} of $\partial \CC M$ is a measured lamination that associates to each transverse arc $c$ an integral sum of the exterior dihedral angles along the leaves that $c$ meets. A simple example to keep in mind arises when $\mu$ is a rational lamination. In this case the geodesic lamination $\lambda$ is the union of a finite number of disjoint simple closed geodesics $\gamma_i$, and $\mu$ is a weighted sum $\sum_i \theta_i \delta_{\gamma_i}$, where $\theta_i \in (0.\pi)$ and $\delta_{\gamma_i}$ is the transverse measure that counts the geometric intersection with the curve $\gamma_i$. For a more detailed description we refer to \cite[Section~II.1.11]{canary_epstein_green2006} (see also Section \ref{section:derivative_length} for alternative definitions of these objects).
    
    \begin{definition} \label{def:equidistant_and_neigh}
        If $A$ is a subset of a metric space $(X,d)$, the \emph{$\varepsilon$-\hsk neighborhood of $A$ in $X$}, which will be denoted by $\neigh_\varepsilon A$, is the set of points of $X$ at distance $\leq \varepsilon$ from $A$. The \emph{$\varepsilon$-\hsk surface of $A$ in $X$}, which will be denoted by $\surf_\varepsilon A$, is the set of points of $X$ at distance $\varepsilon$ from $A$.
    \end{definition}
    
    \begin{remark} \label{rmk:equidistant_surface_C11}
        If $C$ is a closed convex subset in $\Hyp^3$, then the surfaces $\surf_\varepsilon C$ are strictly convex $\mathscr{C}^{1,1}$-\hsk surfaces. Indeed, the distance function $\mappa{d(C, \cdot)}{\Hyp^3}{\R_{\geq 0}}$ is continuously differentiable on $\Hyp^3 \setminus C$ (see \cite[Lemma~II.1.3.6]{canary_epstein_green2006}) and its gradient is uniformly Lipschitz on
        \[
        \overline{\neigh_\varepsilon C \setminus \neigh_{\varepsilon'} C}
        \]
        for all $\varepsilon > \varepsilon' > 0$ (see \cite[Section~II.2.11]{canary_epstein_green2006}). In particular, the equidistant surfaces from the convex core of a convex co-compact hyperbolic manifold $M$ are $\mathscr{C}^{1,1}$-\hsk surfaces.
    \end{remark}
    
    Let $\Sigma$ be a surface immersed in a Riemannian $3$-manifold $X$. The \emph{first fundamental form} $\I$ of $\Sigma$ is the symmetric $(2,0)$-\hsk tensor obtained as pullback of the metric on $X$. Given a choice of a normal vector field $\nu$, the \emph{shape operator} of $\Sigma$ is the $\I$-\hsk self-adjoint $(1,1)$-\hsk tensor $B$, defined by setting $B U \defin - \altmathcal{D}_U \nu$, where $\altmathcal{D}$ is the Levi-Civita connection of $X$ and $U$ is a tangent vector field to $\Sigma$. The \emph{second fundamental form}, denoted by $\II$, is the symmetric $(2,0)$-\hsk tensor $\II(V,W) \defin \I(B V, W) = \I(V, B W)$, for any tangent vector fields $V$, $W$ to $\Sigma$. The \emph{mean curvature} $H$ is the trace of $B$. The notions of second fundamental form, shape operator and mean curvature depend on the choice of a normal vector field on $\Sigma$. Wherever we have to deal with surfaces which are boundaries of domains or with portions of $\varepsilon$-\hsk surfaces, we will always endow them with the interior normal vector field pointing toward the domain or the $\varepsilon$-\hsk neighborhood, respectively.
    
    \emph{Lines} and \emph{planes} in $\Hyp^3$ are $1$ and $2$-\hsk dimensional totally geodesic subspaces of $\Hyp^3$, respectively. A \emph{half-space} is the closure on one of the complementary regions of a plane inside $\Hyp^3$. In the following we recall the geometric data of the equidistant surfaces from a plane and a line, respectively. For a proof of them, we refer for instance to \cite[Chapter~II.2]{canary_epstein_green2006}.
    
    \begin{lemma} \label{lemma:equid_surface_plane}
        Let $P$ be a plane in $\Hyp^3$, and fix $\nu$ a unit normal vector field on $P$. Then the map $\mappa{\eta_\varepsilon}{P}{\Hyp^3}$, defined by
        \[
        \eta_\varepsilon(p) \defin \exp_p(\varepsilon \nu(p)) ,
        \]
        parametrizes a connected component of the $\varepsilon$-\hsk surface from the hyperbolic plane $P$ in $\Hyp^3$, and in these coordinates we have
        \begin{gather*}
            \I_\varepsilon = \cosh^2 \varepsilon \, g_P , \\
            \II_\varepsilon = \frac{\sinh 2 \varepsilon}{2} \, g_P = \tanh \varepsilon \, \I_\varepsilon ,
        \end{gather*}
        where we are choosing as unit normal vector field the one pointing toward the $\varepsilon$-\hsk neighborhood of $P$.
    \end{lemma}
    
    \begin{lemma} \label{lemma:equid_surface_line}
        Let $\mappa{\tilde{\gamma}}{\R}{\Hyp^3}$ be a unit speed complete geodesic, and denote by $e_1(s)$, $e_2(s)$ the vectors, tangent at $\tilde{\gamma}(s)$, obtained as parallel translations of a fixed orthonormal basis $e_1$, $e_2$ of $\tilde{\gamma}'(0)^\perp \subset T_{\tilde{\gamma}(0)} \Hyp^3$. Then the map $\mappa{\psi_\varepsilon}{\R \times S^1}{\Hyp^3}$, defined by
        \[
        \psi_\varepsilon(s, e^{i \theta}) \defin \exp_{\tilde{\gamma}(s)}(\varepsilon (\cos \theta \, e_1(s) + \sin \theta \, e_2(s))) ,
        \]
        parametrizes the $\varepsilon$-\hsk surface from the line $\tilde{\gamma}$ and in these coordinates we have
        \begin{align*}
            \I_\varepsilon & = \cosh^2 \varepsilon \dd{s}^2 + \sinh^2 \varepsilon \dd{\theta}^2 , \\
            \II_\varepsilon & = \cosh \varepsilon \sinh \varepsilon \, (\dd{s}^2 + \dd{\theta}^2) ,
        \end{align*}
        where we are choosing as unit normal vector field the one pointing toward the $\varepsilon$-\hsk neighborhood of $\tilde{\gamma}$.
    \end{lemma}
    
    We want to give a more precise description of the structure of the boundary of the convex core and, to do so, we need to recall the following notion:
    
    \begin{definition}[\cite{bonahon1996shearing}]
        Let $S$ be a topological surface. A \emph{(abstract) pleated surface} with topological type $S$ is a pair $(\tilde{f},\rho)$, where $\mappa{\tilde{f}}{\widetilde{S}}{\Hyp^3}$ is a continuous map from the universal cover $\widetilde{S}$ of $S$ to $\Hyp^3$ and $\mappa{\rho}{\pi_1(S)}{\Iso^+(\Hyp^3)}$ is a homomorphism, verifying the following properties: 
        \begin{enumerate}
            \item $\tilde{f}$ is $\rho$-\hsk equivariant;
            \item the path metric on $\widetilde{S}$, obtained by pullback of the metric on $\Hyp^3$ under $\tilde{f}$, induces a hyperbolic metric $m$ on $S$;
            \item there exists a $m$-geodesic lamination on $S$ such that $\tilde{f}$ sends every leaf of the preimage $\tilde{\lambda} \in \widetilde{S}$ in a geodesic of $\Hyp^3$ and such that $\tilde{f}$ is totally geodesic on each complementary region of $\tilde{\lambda}$ in $\widetilde{S}$.
        \end{enumerate}
    \end{definition}
    
    Consider $\widetilde{\CC}$ the preimage of $\CC M$ inside $\Hyp^3 \cong \widetilde{M}$. The boundary $\partial \widetilde{\CC}$ is parametrized by a pleated surface $\mappa{\tilde{f}}{\widetilde{S}}{\Hyp^3}$ with bending locus $\tilde{\lambda}$, where $\widetilde{S}$ is the universal cover of $\partial \CC M$, and with holonomy $\rho$ given by the composition of the homomorphism induced by the inclusion $\partial \CC M \rightarrow M$ and the holonomy representation of $M$. In this situation, the pleated surface $\tilde{f}$ is \emph{locally convex}, in the sense that the bending occurs always in the same direction, making $\tilde{f}$ locally bound a convex region (see also \cite[Section~II.1.11]{canary_epstein_green2006}). In general $\tilde{f}$ is a covering of $\partial \widetilde{\CC}$, which is non-trivial whenever $\CC M$ has compressible boundary.
    
    It will be useful in our analysis to have a way to locally approximate $\partial \CC M$ by finitely bent surfaces. We briefly recall a procedure described in \cite[Section~7]{bonahon1996shearing} which is well suited for our purpose. We start by considering an arc $k$ in $\widetilde{S}$ transverse to the bending lamination $\tilde{\lambda}$, having endpoints in two different flat pieces $P$ and $Q$ of $\widetilde{S} \setminus \tilde{\lambda}$. We will assume $k$ to be short enough, so that we can find an open neighborhood $U$ of $k$ on which $\tilde{f}$ is a topological embedding, and all the leaves of $\tilde{\lambda}$ meeting $U$ intersect $k$. When this happens, we say that $\tilde{f}$ is a \emph{nice embedding near $k$}. Let $\altmathcal{P}_{P Q}$ be the set of those flat pieces in $\widetilde{S} \setminus \tilde{\lambda}$ that separate $P$ from $Q$. For every finite subset $\altmathcal{P}$ of $\altmathcal{P}_{P Q}$, we label its elements by $P_0, \dots, P_{n + 1}$ following the order from $P = P_0$ to $Q = P_{n + 1}$. Let $\Sigma_i$ be the closure of the region in $\widetilde{S}$ which lies between $P_i$ and $P_{i + 1}$, for $i = 0, \dots, n$. If we orient the two leaves $\gamma_i$ $\gamma_i'$ lying in $\partial \Sigma_i$ accordingly, so that they can be deformed continuously from one to the other though oriented geodesics in $\Sigma_i$, then we call \emph{diagonals} of $\Sigma_i$ the two unoriented lines in $\Sigma_i$ that connect two opposite endpoints of $\gamma_i$ and $\gamma_i'$.
    
    We denote by $\tilde{\lambda}_\altmathcal{P}$ the geodesic lamination of $\widetilde{S}$ obtained from $\tilde{\lambda}$ as follows: we maintain the geodesic lamination as it is outside $\bigcup_i \Sigma_i$ and, for every $i = 0, \dots, n$, we erase all the leaves lying in the interior of the strip $\Sigma_i$ and we replace them by one of the two diagonals of $\Sigma_i$, say $d_i$. Now we define a pleated surface $\mappa{\tilde{f}_\altmathcal{P}}{\widetilde{S}}{\Hyp^3}$, with bending locus $\tilde{\lambda}_\altmathcal{P}$, so that it coincides with $\tilde{f}$ outside the strips, and inside any $\Sigma_i$ it sends the chosen $d_i$ in the geodesic of $\Hyp^3$ joining the endpoints of $\tilde{f}(\partial \Sigma_i)$ corresponding to the endpoints of $d_i$. Once we make a choice of a diagonal $d_i$ for any $i$, there is a unique way to extend $\tilde{f}_\altmathcal{P}$ on $\widetilde{S}$ so that is becomes a pleated surface bent along $\tilde{\lambda}_\altmathcal{P}$. Moreover, if the strips $\Sigma_i$ are thin enough and if the starting $\tilde{f}$ is locally convex, then we can make a choice of the diagonals $d_0, \dots, d_n$ so that the resulting $\tilde{f}_\altmathcal{P}$ is still locally convex. Such $\tilde{f}_\altmathcal{P}$ will not be equivariant anymore under the action of the holonomy of $\tilde{f}$, but it will approximate the restriction of $\tilde{f}$ on $U$.
    
    Now, choose a sequence of increasing subsets $\altmathcal{P}_n$ exhausting $\altmathcal{P}_{P Q}$ and construct a corresponding sequence of convex pleated surfaces $\tilde{f}_n \defin \tilde{f}_{\altmathcal{P}_n}$ as above. Every such $\tilde{f}_n$ is finitely bent on the neighborhood $U$. Following the construction, we see that, given any $P'$ flat piece of $\widetilde{S}$ intersecting $k$, there exists a large $N \in \N$ so that $\tilde{f}_n(P') = \tilde{f}(P') \subset \partial \widetilde{\CC}$ for every $n \geq N$. In particular, the functions $\tilde{f}_n$ are approximating $\tilde{f}$ over the open set $U$. Moreover, following the proof of \cite[Lemma~22]{bonahon1996shearing}, we see that the bending measures $\mu_n(k)$ of $\tilde{f}_n$ on the arc $k$ are converging to $\mu(k)$, the bending measure of $k$ in $\partial \widetilde{\CC}$.
    
    Let now $\mappa{r}{\Hyp^3}{\widetilde{\CC}}$ denote the metric retraction of $\Hyp^3$ over the convex set $\widetilde{\CC}$ and let $\mappa{d}{\Hyp^3}{\R_{\geq 0}}$ be the distance from $\widetilde{\CC}$. We select an open neighborhood $V$ of $k$ so that $\overline{V} \subset U$ and, fixed $\rho > 0$, we define $W = W(V,\rho) \defin r^{-1}(V) \cap \neigh_\rho \widetilde{\CC}$. The surfaces $\tilde{f}_n(U)$ lie behind $\tilde{f}(U) \subset \partial \widetilde{\CC}$ if seen from $W$. Denote by $\mappa{d_n}{W}{\R_{\geq 0}}$ the distance function from $\tilde{f}_n(U)$ on $W$. Since the surfaces $\tilde{f}_n(U)$ are convex, for every point $p \in W$ there exists a unique $q_n \in \tilde{f}_n(U)$ realizing $d_n(p) = d(p,q_n)$. Therefore, it makes sense to consider the metric retractions $\mappa{r_n}{W}{\tilde{f}_n(U)}$, which will converge to $r$ over the compact sets of $W$ thanks to the convergence properties previously observed of the $\tilde{f}_n$'s. By the same argument as \cite[Lemma~II.2.11.1]{canary_epstein_green2006}, the distance functions $d_n$ are converging $\mathscr{C}^{1,1}$-uniformly to $d$ on any compact set of $W$ (i. e. the gradients $\grad d_n$ are uniformly Lipschitz and they converge to $\grad d$). This shows that for every $\varepsilon < \rho$, the surface $d^{-1}(\varepsilon) \cap W = \surf_\varepsilon \widetilde{\CC} \cap W$ is $\mathscr{C}^{1,1}$-\hsk approximated by the sequence of surfaces $(d_n^{-1}(\varepsilon))_n \subset W$. Moreover, such surfaces $d_n^{-1}(\varepsilon) \subset W$ are the $\varepsilon$-\hsk equidistant surfaces from finitely bent convex pleated surfaces having bending measures on $k$ converging to $\mu(k)$.
    
    \begin{definition} \label{def:standard_approx}
        Given $k$ an arc on which $\tilde{f}$ is a nice embedding, we say that the sequence $\tilde{f}_n$ defined above is a \emph{standard approximation of $\partial \widetilde{\CC}$ near $k$} and that the the sequence of surfaces $\surf_{\varepsilon, n}$ is a \emph{standard approximation of $\surf_\varepsilon \widetilde{\CC}$ over $k$}.
    \end{definition}
    
    \section{The dual volume}
    
    This section is devoted to the definition of dual volume on convex sets sitting inside a convex co-compact $3$-manifold, and its main properties.
    
    \begin{definition} \label{def:dual_vol_in_M}
        Let $M$ be a convex co-compact hyperbolic manifold. If $N$ is a compact convex subset of $M$ with $\mathscr{C}^{1,1}$-\hsk boundary, we define the \emph{dual volume} of $N$ as
        \[
        \Vol^*(N) \defin \Vol(N) - \frac{1}{2} \int_{\partial N} H \dd{a}.
        \]
        If $N = \CC M$, then we set $\Vol^*(\CC M) \defin \Vol(\CC M) - \frac{1}{2} \length_m(\mu)$, where $m$ and $\mu$ are the hyperbolic metric and the bending measure of $\partial \CC M$, respectively.
    \end{definition}
    
    \begin{remark}
        When $\partial N$ is only $\mathscr{C}^{1,1}$, the mean curvature function is defined almost everywhere and it belongs to $L^\infty (\partial N)$ (here $\partial N$ is endowed with the measure induced by the Riemannian volume form of its induced metric), in particular the integral $\int_{\partial N} H \dd{a}$ is a well-defined quantity.
    \end{remark}
    
    \noindent There is a relation between the notions of dual volume and of \emph{$W$-volume}, defined in \cite{krasnov2008renormalized} and used to introduce the \emph{renormalized volume} of a convex co-compact hyperbolic manifold. If $N$ is a compact convex subset with $\mathscr{C}^{1,1}$-\hsk boundary in a convex co-compact manifold $M$, the $W$-\hsk volume of $N$ is defined as
    \[
    W(N) \defin \Vol(N) - \frac{1}{4} \int_{\partial N} H \dd{a} = \frac{1}{2} \left( \Vol(N) + \Vol^*(N) \right) .
    \]
    In addition, we mention that in \cite[Lemma~3.3]{bridgeman_brock_bromberg2017} the authors described a way to characterize the quantity $\int_{\partial N} H \dd{a}$ in terms of the \emph{metric at infinity} $\rho_N$ associated to the equidistant foliation $(\surf_\varepsilon N)_\varepsilon$. In this way the definition of dual volume (and of $W$-\hsk volume) can be given without any regularity assumption on $\partial N$. More precisely, they showed that
    \[
    \int_{\partial N} H \dd{a} = \Area(\rho_N) - 2 \Area(\partial N) - 2 \pi \chi(\partial M). 
    \]
    We recall that the mean curvature here is the trace of the shape operator $B$, which is defined using the interior normal vector field to $\partial N$; this explains why the relation above differ by a factor $+ 2$ from the one in \cite{bridgeman_brock_bromberg2017}. In particular, the proof of \cite[Proposition~3.4]{bridgeman_brock_bromberg2017} shows also:
    
    \begin{proposition} \label{prop:Vol^*_continuous}
        The dual volume is continuous on the space of compact convex subsets of $M$ with the Hausdorff topology. 
    \end{proposition}
    
    \noindent In light of this fact, the following Proposition, besides its future usefulness, justifies the definition we gave of $\Vol^*(\CC M)$.
    
    \begin{proposition} \label{prop:dual_volume_order_2}
        Let $M$ be a convex co-compact hyperbolic manifold, with convex core $\CC M$, bending lamination $\mu \in \MesLam(\partial \CC M )$ and hyperbolic metric $m$ on the boundary of $\CC M$. Then, for every $\varepsilon > 0$ we have
        \[
        \Vol^*(\neigh_\varepsilon \CC M) = \Vol^*(\CC M) - \frac{\length_m(\mu)}{4} (\cosh 2 \varepsilon - 1) - \frac{\pi}{2} \abs{\chi(\partial \CC M)} (\sinh 2 \varepsilon - 2 \varepsilon) .
        \]
        As a consequence, we have
        \[
        \Vol^*(\neigh_\varepsilon \CC M) = \Vol^*(\CC M) + O(\abs{\chi(\partial \CC M)},\length_m(\mu);\varepsilon^2).
        \]
        \begin{proof}
            First we study $\Vol(\neigh_\varepsilon \CC M) - \Vol(\CC M)$. Let $\lambda$ be the support of $\mu$ and let $\mappa{r'}{\neigh_\varepsilon \CC M}{\CC M}$ be the restriction of the metric retraction. We divide $\neigh_\varepsilon \CC M \setminus \CC M$ in two regions, $(r')^{-1}(\partial \CC M \setminus \lambda)$ and $(r')^{-1}(\lambda)$.
            
            If $F$ is the interior of a flat piece in $\partial \CC M$, then the portion of $\neigh_\varepsilon \CC M$ which retracts onto $F$ through $r'$ has volume equal to
            \[
            \int_0^\varepsilon \int_F \cosh^2 t \dvol_{\Hyp^2} \dd{t} = \frac{\Area(F)}{2} \left( \frac{\sinh 2 \varepsilon}{2} + \varepsilon \right) ,
            \]
            where we are making use of the coordinates described in Lemma \ref{lemma:equid_surface_plane}. Since the lamination $\lambda$ has Lebesgue measure $0$ inside $\partial \CC M$, the sum of the areas of the flat pieces is $\Area(\partial \CC M) = 2 \pi \abs{\chi(\partial \CC M)}$. Therefore the region in $\neigh_\varepsilon \CC M \setminus \CC M$ which retracts over $\partial \CC M \setminus \lambda$ has volume $\pi \abs{\chi(\partial \CC M)} \left( \frac{\sinh 2 \varepsilon}{2} + \varepsilon \right)$.
            
            Let $D$ be the closed convex subset in $\Hyp^3$ obtained as the intersection of two half-spaces whose boundary planes meet with an exterior dihedral angle equal to $\theta_0$ and select $\gamma$ a geodesic arc lying inside the line along which $\partial D$ is bent. Then, the region in $\neigh_\varepsilon D$ which retracts over $\gamma$ has volume equal to
            \begin{equation} \label{eq:volume_over_bending}
                \int_0^\varepsilon \int_0^{\theta_0} \int_\gamma \cosh t \sinh t \dd{\length} \dd{\theta} \dd{t} = \frac{\theta_0 \, \length(\gamma)}{4} (\cosh \varepsilon - 1) .
            \end{equation}
            An immediate consequence of this relation is that whenever $\partial \CC M$ is finitely bent, the volume of $(r')^{-1}(\lambda)$ coincides with $\frac{\length_m(\mu)}{4}(\cosh \varepsilon - 1)$, where $m$ is the hyperbolic metric of $\partial \CC M$. In the general case, we can select a suitable covering of $\partial \CC M$ by open sets on which we can apply the standard approximation argument of Definition \ref{def:standard_approx}. With this procedure, it is straightforward to see that the relation $\Vol((r')^{-1}(\lambda)) = \frac{\length_m(\mu)}{4}(\cosh \varepsilon - 1)$ still holds in the general case. Combining the relations we found, we obtain
            \[
            \Vol(\neigh_\varepsilon \CC M \setminus \CC M) = \pi \abs{\chi(\partial \CC M)} \left( \frac{\sinh 2 \varepsilon}{2} + \varepsilon \right) + \frac{\length_m(\mu)}{4} ( \cosh 2 \varepsilon - 1 ) .
            \]
            
            Now we want to compute $\int_{\surf_\varepsilon \CC M} H_\varepsilon \dd{a}_\varepsilon$. Using Lemmas \ref{lemma:equid_surface_line} and \ref{lemma:equid_surface_plane} we immediately see that, in the finitely bent case the following holds:
            \[
            \int_{\surf_\varepsilon \CC M} H_\varepsilon \dd{a}_\varepsilon = 2 \pi \abs{\chi(\partial \CC M)} \sinh 2 \varepsilon + \length_m(\mu)\cosh 2 \varepsilon .
            \]
            The standard approximation procedure (see Definition \ref{def:standard_approx}) allows us again to prove this relation in the general case, with the only difference that the $\mathscr{C}^{1,1}$-\hsk convergence is now crucial, because the expression of the mean curvature in chart involves the second derivatives in the coordinates system. Combining the relations we proved with the equality $\Vol^*(\CC M) = \Vol(\CC M) - \length_m(\mu)/2$, we deduce the relation in the statement.
        \end{proof}
    \end{proposition}
    
    As we will see in a moment, it will be convenient for us to differentiate the dual volume enclosed in a differentiable $1$-\hsk parameter family of $\mathscr{C}^{1,1}$-\hsk surfaces. In particular, we will make use of the following result, which is a corollary of the differential Schl\"afli formula proved in \cite{schlenker_rivin1999schlafli}:
    
    \begin{proposition} \label{prop:variation_formula_dual_volume_in_M}
        Let $N$ be a compact manifold with boundary and let $M_t = (N,g_t)$ be a smooth $1$-parameter family of complete convex co-compact hyperbolic structures on $N \setminus \partial N$. Let now $C$ be a compact set of $N \setminus \partial N$ with $\mathscr{C}^{1,1}$-\hsk boundary and assume that $N$ is convex with respect to the structure $M_t$ for all small values of $t$. Then the variation of the dual volume of $(C,g_t)_t$ at $t = 0$ exists and can be expressed as
        \[
        \left. \dv{\Vol^*(C,g_t)}{t} \right|_{t = 0} = \frac{1}{4} \int_{\partial C} \scall{\delta \I}{\II - H \I} \dd{a} ,
        \]
        where $\I$, $\II$, $H$ are the first and second fundamental forms and the mean curvature of the surface $\partial N_0$, and $\scall{\cdot}{\cdot}$ is the scalar product induced by $\I$ on the space of $2$-\hsk tensors on $\partial N_0$.
        \begin{proof}
            
            By \cite[Theorem~8]{schlenker_rivin1999schlafli}, the variation of the volume of $C$ can be expressed as follows:
            \[
            \left. \dv{\Vol(C,g_t)}{t} \right|_{t = 0} = \frac{1}{2} \int_{\partial C} \left( \delta H + \frac{1}{2} \scall{\delta \I}{\II} \right) \dd{a} .
            \]
            On the other side, the variation of the integral of the mean curvature coincides with
            \[
            \int_{\partial C} \left( \delta H \dd{a} + H \delta(\dd{a}) \right) = \int_{\partial C} \left( \delta H + \frac{H}{2} \scall{\delta \I}{\I} \right) \dd{a} . 
            \]
            Here we used the fact that $\delta(\dd a) = \frac{1}{2} \scall{\delta \I}{\I} \dd{a}$ (this can be proved by differentiating the expression of $\dd{a}$ in local coordinates $(x,y)$, i. e. $\dd{a} = \sqrt{\det(I)} \dd{x} \wedge \dd{y}$). Combining these two variation formulas with the definition of the dual volume we obtain the desired expression for the derivative of the dual volume.
        \end{proof}
    \end{proposition}
    
    Contrary to the case of the hyperbolic volume, it is not clear whether the dual volume of a convex set is positive or not. However, $\Vol^*$ shares, with the usual notion of volume, the property of being monotonic (in fact \emph{decreasing}) with respect to the inclusion, as we see in the following:
    
    \begin{proposition} \label{prop:Vol^*_monotonic}
        Let $N$, $N'$ be two compact convex subsets inside a convex co-compact manifold $M$. If $N \subseteq N'$, then $\Vol^*(N) \geq \Vol^*(N')$.
        \begin{proof}
            Thanks to Proposition \ref{prop:Vol^*_continuous}, up to considering $\varepsilon$-\hsk neighborhoods and passing to the limit as $\varepsilon$ goes to $0$, we can assume that $N$ and $N'$ are compact convex subsets with $\mathscr{C}^{1,1}$-boundary. We will make use of the variation formula of Proposition \ref{prop:variation_formula_dual_volume_in_M}. Assume that $\mappa{\Sigma}{I \times S}{M}$ is a differentiable $1$-\hsk parameter family of convex $\mathscr{C}^{1,1}$-\hsk surfaces $\Sigma_t \defin \Sigma(t, \cdot)$, which parametrize the boundaries of an increasing family of compact convex subsets $(N_t)_{t \in I}$ inside $M$. Let $V_t$ be the infinitesimal generator of the deformation at time $t$, i. e. $V_t$ is the vector field over $S$ defined by $V_t \defin \dv{\Sigma_t}{t}$. The tangential component of $V_t$ does not contribute to the variation of the dual volume (compare with \cite[Theorem~1]{schlenker_rivin1999schlafli}). Consequently, in order to compute the derivative of $\Vol^*(N_t)$, we can assume $V_t$ to be along the exterior normal vector field $- \nu_t$ of $\partial N_t$. Moreover, since the deformation $(N_t)_t$ is increasing with respect to the inclusion, $V_t$ is of the form $- f_t \nu_t$, for some $\mappa{f_t}{S}{\R}$, $f_t \geq 0$. Under this condition, the variation of the first fundamental form of $\partial N_t$ is $\delta \I_t = 2 f_t \II_t$ (again, compare with \cite[Theorem~1]{schlenker_rivin1999schlafli}). If $k_{1,t}$, $k_{2,t}$ denote the principal curvatures of $\partial N_t$, we obtain that
            \begin{align*}
                \scall{\delta \I_t}{H_t \I_t - \II_t} & = 2 f_t \scall{\II_t}{H_t \I_t - \II_t} \\
                & = 2 f_t ((k_{1,t} + k_{2,t})^2 - k_{1,t}^2 - k_{2,t}^2) \\
                & = 4 f_t k_{1,t} k_{2,t} \geq 0,
            \end{align*}
            where, in the last step, we used the fact that the extrinsic curvature $K^e_t = k_{1,t} k_{2,t}$ is non-negative since $\partial N_t$ is convex. By Proposition \ref{prop:variation_formula_dual_volume_in_M}, we deduce that $\Vol^*$ is non-increasing along the deformation $(N_t)_t$. 
            
            It remains to show that, if $N$, $N'$ are two convex subsets of $M$ with $\mathscr{C}^{1,1}$-\hsk boundary and such that $N \subseteq N'$, we can find a differentiable $1$-parameter family, indexed by $t \in [0,1]$, of increasing convex subsets $N_t$ with $\mathscr{C}^{1,1}$-\hsk boundary so that $N_0 = N$ and $N_1 = N'$. A way to produce such a path is described in the proof of \cite[Lemma~3.14]{schlenker2013renormalized}, we briefly recall the ideas involved in the construction. Given any convex set $N$ with $\mathscr{C}^{1,1}$-\hsk boundary in $M$, the asymptotic expansion of the first fundamental forms of the equidistant surfaces from $N$ determines a unique Riemannian metric $h_N$ belonging to the conformal class at infinity of $\partial_\infty M$. Moreover, the surface $\partial N$ can be recovered from $h_N$ as the envelope of a family of horoballs determined by $h_N$, thanks to a construction due to Epstein ($\partial N$ is the so-\hsk called \emph{Epstein surface} associated to the metric $h_N$, see \cite{epstein1984envelopes}). This correspondence behaves well with respect to the inclusion, in the sense that if $N$ and $N'$ are convex sets as above and $N \subseteq N'$, then $h_N \leq h_{N'}$. Being $h_N$ and $h_{N'}$ elements of the same conformal class, there exists a non-\hsk negative function $u$ on $\partial_\infty M$ such that $h_{N'} = e^{2 u} h_N$. If we set now $h_t \defin e^{2 t u} h_N$, then the Epstein surfaces associated to $h_t$ turn out to be the boundaries of an increasing family of convex subsets $N_t$ satisfying the desired requirements (see \cite[Lemma~3.14]{schlenker2013renormalized} for a more detailed exposition).
        \end{proof}
    \end{proposition}
    
    \section{The derivative of the length} \label{section:derivative_length}
    
    From now on, $S$ will be a fixed closed surface of genus $g \geq 2$. We briefly recall the notions of \cite{bonahon1988the_geometry} that we will need. Given $m$ a hyperbolic metric on $S$, the universal cover $\widetilde{S}$, endowed with the lifted metric $\tilde{m}$, is isometric to $\Hyp^2$. As the topological boundary of the Poincar\'e disk sits at infinity of $\Hyp^2$, also $\widetilde{S}$ can be compactified by adding a topological circle $\partial_\infty \widetilde{S}$ at infinity, and the resulting space does not depend on the chosen identification between them. The fundamental group naturally acts by isometries on $\widetilde{S} \cong \Hyp^2$, and since the isometries of $\Hyp^2$ extend to $\partial \Hyp$, the action extends to $\partial_\infty \widetilde{S}$. It turns out that the topological space $\partial_\infty \widetilde{S}$, together with its action of $\pi_1(S)$, is independent of the hyperbolic metric $m$ we chose. In particular, all the spaces we are going to describe are intrinsically associated to the topological surface $S$, without prescribing any additional structure. Since a geodesic in $\widetilde{S}$ is determined by its (distinct) endpoints in $\partial_\infty \widetilde{S}$, the space $\altmathcal{G}(\widetilde{S})$ of unoriented geodesics of $\widetilde{S}$ can be naturally identified with
    \[
    \faktor{(\partial_\infty \widetilde{S} \times \partial_\infty \widetilde{S} \setminus \Delta)}{\Z_2} ,
    \]
    where $\Delta$ denotes the diagonal subspace of $(\partial_\infty \widetilde{S})^2$, and the action of $\Z_2$ exchanges the two coordinates in $(\partial_\infty \widetilde{S})^2$. Therefore, a \emph{geodesic lamination} $\lambda$ of $S$ is identified with a closed, $\pi_1(S)$-\hsk invariant subset $\widetilde{\lambda}$ of disjoint geodesics in $\altmathcal{G}(\widetilde{S})$. In the same spirit, a \emph{measured lamination} of $S$ corresponds to a $\pi_1(S)$-invariant, locally finite Borel measure on $\altmathcal{G}(\widetilde{S})$ with support contained in a geodesic lamination $\lambda$ of $S$. We denote by $\GeoLam(S)$ and $\MesLam(S)$ the spaces of geodesic laminations and measured laminations on $S$, respectively.
    
    \vspace{0.5cm}
    
    In the following, we recall the notion of length of measured laminations realized inside a fixed hyperbolic $3$-manifold $M$ from \cite[Section~7]{bonahon1997geodesic}. As in the case of $S$, we can define the space of unoriented geodesics of $M$, making use of the natural compactification of $\Hyp^3$. The substantial difference is that the dynamical properties of the action of $\pi_1(M)$ do depend in general on the hyperbolic metric we are considering on $M$. However, our interest will be to apply these notions to quasi-isometric deformations of hyperbolic manifolds. In this case, the holonomy representations turn out to be quasi-conformally conjugated in $\partial \Hyp^3$, therefore the qualitative properties of the action of $\pi_1(M)$ on $\altmathcal{G}(\widetilde{M})$ are preserved. Fix now a homotopy class of maps $[\mappa{f_0}{S}{M}]$.
    
    \begin{definition}
        A geodesic lamination $\lambda$ on $S$ is \emph{realizable} inside $M$ in the homotopy class $[f_0]$ if there exists a representative $\mappa{f}{S}{M}$ of $[f_0]$ which sends each geodesic of $\lambda$ homeomorphically in a geodesic of $M$. In such case, we say that $\lambda$ is \emph{realized} by $f$.
    \end{definition}
    
    In order to talk about the realization of a \emph{measured} lamination $\mu$, we need to find a way to push-forward the measure $\mu$ to a measure on $\altmathcal{G}(\widetilde{M})$. Let $\lambda$ be a geodesic lamination on $S$ realized by a map $f$, and let $\mappa{\rho}{\pi_1(S)}{\pi_1(M)}$ be the homomorphism induced by $[f_0]$ on the fundamental groups. Fixed a lift $\tilde{f}$ of $f$ to the universal covers, we can construct a function $\mappa{r}{\tilde{\lambda}}{\altmathcal{G}(\widetilde{M})}$, associating to each leaf $g$ of $\tilde{\lambda}$ the geodesic $\tilde{f}(g)$ sitting inside $\widetilde{M}$. The map $r$ is $\rho$-\hsk equivariant and continuous with respect to the topologies of $\tilde{\lambda}$ as subset of $\altmathcal{G}(\widetilde{S})$ and of $\altmathcal{G}(\widetilde{M})$ (compare with \cite[Section~7]{bonahon1997geodesic}). It is easy to prove that $r$ depends only on the homotopy class $[f]$ and on the choice of a lift of \emph{any} representative of $[f]$ realizing $\lambda$. To see this, let $F_0 = f$ and $f_1 = f'$ be two such maps in $[f]$ homotopic through $(F_t)_{t \in I}$ (here $I$ denotes the interval $[0,1]$). Once we choose a lift $\tilde{f}$ of $f$, there exists a unique lift $\widetilde{F}_t$ of the homotopy so that $\widetilde{F}_0 = \tilde{f}$. This gives a preferred lift of $f'$, namely $\tilde{f}' \defin \widetilde{F}_1$. Because of the compactness of $S$ and the existence of a homotopy $\widetilde{F}_t$ between them, the lifts $\tilde{f}$ and $\tilde{f}'$ must agree (up to reparametrization) on any leaf $g$ of $\tilde{\lambda}$, since the geodesics $\tilde{f}(g)$ and $\tilde{f}'(g)$ are necessarily at bounded distance in $\Hyp^3$ (see \cite[Proposition~8.10.2]{thurston1979geometry}). This implies that the definitions of $r$ obtained using $\tilde{f}$ and $\tilde{f}'$ coincide. Moreover, different choices of lifts $\tilde{f}$ produce maps $r$, $r'$ which differ by post-composition by an element in $\pi_1(M)$. The same argument as above shows that, if $\lambda_1$, $\lambda_2$ are two geodesic laminations realized by the maps $f_1$, $f_2$ respectively, which both contain the lamination $\lambda$, then the two realizations $f_1$ and $f_2$ coincide on $\lambda$.
    
    We are finally ready to describe the definition of the length of the realization of a measured lamination inside $M$. Let $\alpha$ be a measured lamination on $S$ with support contained in $\lambda$. We denote by $\bar{\alpha} \defin r_* \alpha$ the push-forward of $\alpha$ under the map $r$. $\bar{\alpha}$ is a measure on $\altmathcal{G}(\widetilde{M})$ with support $r(\supp \alpha)$, depending only on $\alpha \in \MesLam(S)$, on the homotopy class $[f]$ and on the choice of a lift of $f$. Assume that $f(\lambda)$ lies inside some compact set $K$ of $M$ and let $\altmathcal{F}$, $\widetilde{\altmathcal{F}}$ denote the geodesic foliations of the projective tangent bundles $\Proj T M$, $\Proj T \widetilde{M}$, respectively. We can cover the preimage of $K$ in $\Proj TM$ by finitely many $\altmathcal{F}$-\hsk flow boxes $\mappa{\sigma_j}{D_j \times I}{B_j}$. Here $D_j$ is some topological space and $\sigma_j$ is a homeomorphism sending each subset $\set{p} \times I \subset D_j \times I$ in a subarc of a leaf in $\altmathcal{F}$, for any $p \in D_j$. In addition, we fix a collection $\set{\xi_j}_j$ of smooth functions with supports $\supp \xi_j$ contained in the interior of $B_j$ for every $j$, and such that $\sum_j \xi_j = 1$ over the preimage of $K$ in $\Proj TM$. If $\sigma_j$ is a $\altmathcal{F}$-\hsk flow box that meets $f(\supp \alpha)$, we can lift it to a $\widetilde{\altmathcal{F}}$-\hsk flow box $\mappa{\tilde{\sigma}_j}{D_j \times I}{\Proj T \widetilde{M}}$ accordingly with the choice of the lift $\tilde{f}$. The lift $\tilde{\sigma}_j$ induces an identification between the space $D_j$ with a subset in $\altmathcal{G}(\widetilde{M})$. Namely, a point $p \in D_j$ corresponds to the complete leaf in $\widetilde{\altmathcal{F}}$ extending the arc $\tilde{\sigma}_j(\set{p} \times I)$. Through this identification, it makes sense to integrate the $D_j$-component of $\tilde{\sigma}_j$ with respect to the measure $\bar{\alpha}$ previously defined on $\altmathcal{G}(\widetilde{M})$. If $\sigma_j$ does not meet $f(\supp \alpha)$, then we choose an arbitrary lift $\tilde{\sigma}_j$. Finally, we select lifts $\tilde{\xi}_j$'s of the $\xi_j$'s according with the choices of the lifts $\tilde{\sigma}_j$.  The \emph{length of the realization of $\alpha$ in $M$} (in the homotopy class $[f]$) is
    \begin{equation} \label{eq:def_lenght}
        \length_M(\alpha) = \iint_\lambda \dd{\length} \dd{\alpha} \defin \sum_j \int_{D_j} \int_0^1 \tilde{\xi}_j(\tilde{\sigma}_j(p,s)) \dd{\length}(s) \dd{\bar{\alpha}}(p) ,
    \end{equation}
    where $d\length$ denotes the length-measure along the leaves of $\widetilde{\altmathcal{F}}$.
    
    \begin{remark}
        By invariance of the length under reparametrization and by linearity of the integral, the choices of the functions $\set{\xi_j}_j$ and the chosen $\altmathcal{F}$-\hsk flow boxes $\set{\sigma_j}_j$ are irrelevant; moreover, different lifts of $f$ produce maps $r$ which are conjugated by isometries in $\pi_1(M)$. Therefore, the quantity $\length_M(\alpha)$ only depends on the measured lamination $\alpha$, the hyperbolic metric on $M$ and the homotopy class $[\mappa{f}{S}{M}]$. The notion makes sense as long as there exists a realizable geodesic lamination $\lambda$ in the homotopy class $[f]$ which contains $\supp \alpha$. Moreover, by what we observed before, this quantity does not depend on the specific representable lamination $\lambda$ we chose, but it is determined only by $\supp \alpha$.
    \end{remark}
    
    We are now ready to produce a variation formula for the length of the realization of a measured lamination inside a $1$-\hsk parameter family of quasi-isometric convex co-compact hyperbolic manifolds $(M_t)_t$. For convenience, we think of $(M_t)_t$ as a differentiable $1$-\hsk parameter family of complete hyperbolic metrics $g_t$ on a fixed $3$-\hsk manifold $X$, so that the identity map, from $M = M_0 = (X,g_0)$ to $M_t = (X,g_t)$, is a quasi-isometric diffeomorphism for any $t$. Let $\alpha \in \MesLam(S)$ be a measured lamination and $[\mappa{f_0}{S}{X}]$ a homotopy class of maps. In the convex co-compact case, all finite laminations are realizable and their realizations are necessarily contained in the convex core $\CC M_t$. Therefore, by \cite[Corollary~I.5.2.13]{canary_epstein_green2006} and \cite[Theorem~I.5.3.6]{canary_epstein_green2006}, any geodesic lamination on $S$ is realizable in the homotopy class $[f_0]$, and their realizations lie inside a fixed compact subset $K$ of $X$ (where $K$ contains $\CC M_t$ for every small $t$). Let now $\lambda$ be \emph{any} geodesic lamination containing $\supp \alpha$ and assume that it is realized inside $M_t$ by a certain map $\mappa{f_t}{S}{M_t}$, for any $t$. By the above, we are allowed to consider the length of the realization of $\alpha$ inside $M_t$ for every $t$. Let $\set{\sigma_j}_j$, $\set{\xi_j}_j$, $\set{\tilde{\sigma}_j}_j$, $\set{\tilde{\xi}_j}_j$ be a collection of functions as in the definition of $\length_M(\alpha)$. Then, in the same notations as above, we set
    \[
    \iint_\lambda \dd{\dot{\length}} \dd{\alpha} \defin \sum_j \int_{D_j} \int_0^1 \tilde{\xi}_j(\tilde{\sigma}_j(p,s)) \, \frac{\dot{g} \left( \partial_s \tilde{\sigma}_j(p,s), \partial_s \tilde{\sigma}_j(p,s) \right)}{2 g \left( \partial_s \tilde{\sigma}_j(p,s), \partial_s \tilde{\sigma}_j(p,s) \right)} \dd{\length} \dd{\bar{\alpha}}(p) ,
    \]
    where $\partial_s \tilde{\sigma}_j = \pdv{\tilde{\sigma}_j}{s}$, $g = g_0$ and $\dot{g} = \dv{g}{t} |_{t = 0}$. The result we want to prove is the following:
    
    \begin{proposition} \label{prop:variation_length}
        Let $(g_t)_t$ be a $1$-\hsk parameter family of convex co-compact hyperbolic metrics on a $3$-\hsk manifold $X$, which are quasi-isometric to each other via the identity map of $X$. Let $\alpha$ be a measured lamination on a surface $S$ and let $[\mappa{f}{S}{X}]$ be a fixed homotopy class. Then $\alpha$ is realizable in $M_t$ for all values of $t$, and the variation of its length verifies
        \begin{equation}
            \left. \dv{\length_{M_t}(\alpha)}{t} \right|_{t = 0} = \iint_\lambda \dd{\dot{\length}} \dd{\alpha} ,
        \end{equation}
        where $\lambda$ is a geodesic lamination of $S$ containing $\supp \alpha$.
    \end{proposition}
    
    We will prove the Proposition using an approximation argument. Firstly we deal with the rational case:
    
    \begin{lemma} \label{lemma:variation_length_rational_case}
        When $\alpha \in \MesLam(S)$ is a rational lamination, Proposition \ref{prop:variation_length} holds.
        \begin{proof}
            Let $c$ be a free homotopy class of simple closed curves in $X$ and assume that $c$ admits a geodesic representative in $M_0$. Since we are considering a quasi-isometric deformation of convex co-compact manifolds, the homotopy class $c$ will admit a geodesic representative for all values of $t$. Moreover, we can find parametrizations $\gamma^t$ of the geodesic of $c$ in $M_t$ depending smoothly on $t$, because of the smooth dependence of the holonomy representation $\hol_t(c)$. In other words, we can find a smooth map $\mappa{\Sigma}{(- \varepsilon, \varepsilon) \times I}{X}$ such that $\Sigma(t,s) = \gamma^t(s)$ for every $t$ and $s \in I$. Let $\norm{\cdot}_t$ denote the norm with respect to the metric $g_t$, and let $\gamma = \gamma^0$. We have
            \begin{align*}
                \left. \dv{t} \norm{\partial_s \gamma^t}_t \right|_{t = 0} & = \frac{\dot{g} (\partial_s \gamma,\partial_s \gamma) + 2 g( \left. \altmathcal{D}_{\partial_t} \partial_s \Sigma \right|_{t = 0}, \partial_s \gamma)}{2 \norm{\partial_s \gamma}_0} \\
                & = \frac{\dot{g} (\partial_s \gamma,\partial_s \gamma)}{2 \norm{\partial_s \gamma}_0} + g \left( \altmathcal{D}_{\partial_s} \left. \partial_t \Sigma \right|_{t = 0}, \frac{\partial_s \gamma}{\norm{\partial_s \gamma}_0} \right) \\
                & = \frac{\dot{g} (\partial_s \gamma,\partial_s \gamma)}{2 \norm{\partial_s \gamma}_0} + \dv{s} \left[ g \left( \left. \partial_t \Sigma \right|_{t = 0}, \frac{\partial_s \gamma}{\norm{\partial_s \gamma}_0} \right) \right] ,
            \end{align*}
            where in the last step we used the fact that $\gamma$ parametrizes a geodesic in $M = M_0$, and consequently the covariant derivative of $\frac{\partial_s \gamma}{\norm{\partial_s \gamma}_0}$ vanishes. Once we integrate the last term in $t \in [0,1]$ we get $0$, because the function of which we are taking the derivative coincides at the extremes (since the geodesics $\gamma^t$ are closed). Hence we obtain
            \[
            \left. \dv{\length_{M_t}(c)}{t} \right|_{t = 0} = \int_0^1 \frac{\dot{g} (\partial_s \gamma,\partial_s \gamma)}{2 \norm{\partial_s \gamma}_0} \dd{s} = \int_0^1 \frac{\dot{g}(\partial_s \gamma,\partial_s \gamma)}{2 g(\partial_s \gamma,\partial_s \gamma)} \dd{\length} .
            \]
            Take now a rational lamination $\alpha \in \MesLam(S)$, i. e. the measure $\alpha$ is the weighted sum $\sum_i u_i \, \delta_{d_i}$, where the $d_i$ are homotopy classes of simple closed curves, the $u_i$ are positive weights, and $\delta_{d_i}$ is the transverse measure which counts the geometric intersection of an arc transverse to $d_i$ with $d_i$. Assume that $\alpha$ is realizable in $M$ or, equivalently, that the curves $c_i = f_0(d_i)$ admit a geodesic representative $\gamma_i$ in $M$. The same argument given above shows that the lamination $\alpha$ is realizable in $M_t$ for all $t$. Applying the definition of $\length_{M_t}(\alpha)$, and denoting by $\mappa{\gamma^t_i}{I}{M_t}$ the geodesic representative of $c_i$, we see that
            \[
            \length_{M_t}(\alpha) \defin \sum_i u_i \left( \int_0^1 \norm{\partial_s \gamma^t(s)}_t \dd{s} \right) .
            \]
            Hence, taking the derivative in $t$ and using what observed above, we get
            \[
            \left. \dv{\length_{M_t}(\alpha)}{t} \right|_{t = 0} = \sum_i u_i \left( \int_0^1 \frac{\dot{g} (\partial_s \gamma,\partial_s \gamma)}{2 \norm{\partial_s \gamma}_0} \dd{s} \right) = \iint_\lambda \dd{\dot{\length}} \dd{\alpha} ,
            \]
            where $\lambda = \supp \alpha = \bigcup_i d_i$.
        \end{proof}
    \end{lemma}
    
    We are now ready to deal with the proof of Proposition \ref{prop:variation_length}:
    
    \begin{proof}[Proof of Proposition \ref{prop:variation_length}]
        Let $T$ be a train track in $S$ carrying $\alpha$ and consider a sequence of rational laminations $\alpha_n$ carried by $T$ and converging to $\alpha$ as measured laminations (see \cite[Proposition~8.10.7]{thurston1979geometry}). Up to passing to a subsequence, we can assume that the laminations $\supp \alpha_n$ converge in the Hausdorff topology to a lamination $\lambda$ carried by $T$. Since $\alpha_n$ is converging to $\alpha$, we must have $\lambda \supseteq \supp \alpha$. We denote by $\mappa{f_t}{S}{X}$ a realization of $\lambda$ in the homotopy class $[f]$ with respect to the metric $g_t$, and by $\mappa{\tilde{f}_t}{\widetilde{S}}{\widetilde{M}}$ lifts of the $f_t$'s so that $t \mapsto \tilde{f}_t$ is continuous with respect to the compact-open topology of $\mathscr{C}^0(\widetilde{S},\widetilde{X})$.
        
        Let now $K$ be a large compact set of $X$ containing all the convex cores $\CC M_t$ for small values of $t$. Then, if $\altmathcal{F}_t$ is the geodesic foliation of $\Proj M_t$, we can choose $\altmathcal{F}_t$-\hsk flow boxes $\set{\sigma^t_j}_j$ whose union of images contain the preimage of $K$ in $\Proj T M_t$, and hence the realizations $f_t(\lambda)$. We consequently construct maps $\set{\tilde{\sigma}^t_j}_j$, $\set{\xi^t_j}_j$, $\set{\tilde{\xi}^t_j}_j$ as in the definition of $\length_{M_t}(\cdot)$. We can ask these functions to vary smoothly in the parameter $t$, since the hyperbolic metrics depends smoothly in $t$. Now, we define 
        \[
        \varphi_j^t(\cdot) \defin \int_0^1 \tilde{\xi}_j^t(\tilde{\sigma}_j^t(\cdot,s)) \dd{\length}_t(s) .
        \]
        In this notation, the length of the realization of $\alpha_n$ in $M_t$ can be expressed as
        \[
        \length_{M_t}(\alpha_n) = \sum_j \int_{D_j} \varphi^t_j \dd{\bar{\alpha}}_n .
        \]
        From this relation is clear that, as $n$ goes to $\infty$, $\length_{M_t}(\alpha_n)$ converges uniformly to $\length_{M_t}(\alpha)$ on a small interval $(-\varepsilon, \varepsilon)$ of the parameter $t$. In the same way we see that $\iint \dd{\dot{\length}} \dd{\alpha}_n$ converges to $\iint \dd{\dot{\length}} \dd{\alpha}$ (here is even easier, because there is no dependence on $t$). Thanks to Lemma \ref{lemma:variation_length_rational_case}, the only thing left to conclude the proof is to show that
        \[
        \lim_{n \rightarrow \infty} \left. \dv{\length_{M_t}(\alpha_n)}{t} \right|_{t = 0} = \left. \dv{\length_{M_t}(\alpha)}{t} \right|_{t = 0} .
        \]
        Here we can argue as follows: the length of a homotopy class $c$ of non-parabolic type can be expressed as the real part of its \emph{complex length} $\length_\bullet^\C(c) \in \C / 2 \pi i \Z$, which is holomorphic in the holonomy representation. The argument described above shows that the real lengths $\length_\bullet(\alpha_n)$ are converging uniformly in a small neighborhood of $\hol_0$ (see also \cite[Theorem~2]{sullivan1981travaux}). Since the real part of a holomorphic function determines (up to imaginary constant) the holomorphic function itself, we deduce that also the complex lengths $\length_\bullet^\C(\alpha_n)$ are converging uniformly, and hence $\mathscr{C}^\infty$-\hsk uniformly. In particular this proves the convergence of the derivatives in $t$.
    \end{proof}
    
    \section{The dual Bonahon-Schl{\"a}fli formula} \label{sec:dual_bonahon_schlafli}
    
    In this section we will describe the proof of Theorem \ref{thmx:dual_bonahon_schlafli}. As mentioned in the introduction, the first subsection will be dedicated to the study of the convexity of the equidistant surfaces from the convex core while we vary the hyperbolic structure. Afterwards we will introduce an auxiliary function on which we can apply the differential Schl\"afli formula (Proposition \ref{prop:variation_formula_dual_volume_in_M}). This is the step in which the variation of the length of the bending measure arises (see Proposition \ref{prop:derivative_limit_schlafli}). In Proposition \ref{prop:dual_volume_derivative_length} we will relate this with the actual variation of the dual volume of the convex core. In the end of the section we will use Bonahon's results about the dependence of the metric of the convex core in terms of the convex co-compact hyperbolic structure to finally prove Theorem \ref{thmx:dual_bonahon_schlafli}.
    
    \vspace{0.5cm}
    
    Let $(M_t)_t$ be a smooth family of quasi-isometric convex co-compact manifolds, parametrized by $t \in (- t_0, t_0)$. We can choose diffeomorphisms $\mappa{\varphi_t}{M_0}{M_t}$ so that the following properties hold:
    \begin{enumerate}
        \item $\varphi_t$ is a quasi-isometric diffeomorphism for any $t$, and $\varphi_0 = \id$;
        \item fixed identifications of the universal covers of $M_t$ with $\Hyp^3$ for every $t$, we can find lifts $\mappa{\tilde{\varphi}_t}{\Hyp^3}{\Hyp^3}$ of $\varphi_t$ so that $\tilde{\varphi}_0 = \id_{\Hyp^3}$ and so that the map $\tilde{\varphi}$, defined by $\tilde{\varphi}(t, \cdot ) \defin \tilde{\varphi}_t(\cdot)$, is smooth as a map from $(- t_0, t_0) \times \Hyp^3$ to $\Hyp^3$.
    \end{enumerate}
    
    \subsection{Convexity of equidistant surfaces}
    
    In order to prove Theorem \ref{thmx:dual_bonahon_schlafli}, it will be important for us to understand for which values of $t$ and $\varepsilon \leq \varepsilon_0$ the surfaces $\varphi_t(\surf_\varepsilon \CC M_0)$ and $\varphi_t^{-1}(\surf_\varepsilon \CC M_t)$ remain convex. This is the most technical part of our argument and it will require special care. We want to prove the following fact:
    
    \begin{lemma} \label{lemma:key_convexity_property}
        There exist constants $K$, $\tau > 0$, with $0 < \tau \leq t_0$, which depend only on the quasi-isometric deformation $(M_t)_t$ and on the fixed family of diffeomorphims $(\varphi_t)_t$, such that, for every $t \in (- \tau, \tau)$ the regions $\varphi_t(\neigh_{K \abs{t}} \CC M_0)$ and $\varphi^{-1}_t(\neigh_{K \abs{t}} \CC M_t)$ are convex in $M_t$ and $M_0$, respectively. As a consequence, we have
        \[
        \varphi_t(\neigh_{K \abs{t}} \CC M_0) \supset \CC M_t \qquad \text{and} \qquad \neigh_{K \abs{t}} \CC M_t \supset \varphi_t(\CC M_0) .
        \]
    \end{lemma}
    
    We denote by $\mappa{\pi_t}{\Hyp^3}{M_t}$ the universal cover of $M_t$, and by $\widetilde{\CC}_t \subset \Hyp^3$ the preimage of the convex core $\CC M_t$ under $\pi_t$. Fixed $q_0$ a basepoint in $\Hyp^3$, we can find a large $R > 0$ so that the metric ball $B_R = B(q_0,R)$ in $\Hyp^3$ verifies
    \[
    \pi_t \tilde{\varphi}_t(B_R) = \varphi_t \pi_0(B_R) \supseteq \neigh_{\varepsilon_0} \CC M_t
    \]
    and $\varphi_t(\overline{B}_R) \subseteq B_{R + 1}$, whenever $t$ is small enough. 
    This follows from the fact that the convex cores $\CC M_t$ are compact and they vary continuously in the parameter $t$. Clearly Lemma \ref{lemma:key_convexity_property} reduces to the study of the surfaces $\tilde{\varphi}_t(\surf_\varepsilon \widetilde{\CC}_0 \cap B_R)$ and $\tilde{\varphi}_t^{-1}(\surf_\varepsilon \widetilde{\CC}_t \cap B_R)$ in $\Hyp^3$.
    However, instead of dealing directly with equidistant surfaces from $\widetilde{C}_0$, which are only $\mathscr{C}^{1,1}$, we will rather focus our study on the family of $\varepsilon$-\hsk surfaces from half-spaces of $\Hyp^3$, which are more regular and can be used as "support surfaces" for $\surf_\varepsilon \widetilde{C}_0$. The strategy will be to understand how the convexity of their image under $\tilde{\varphi}_t$ behave, and from this to deduce the convexity of the surfaces $\tilde{\varphi}_t(\surf_\varepsilon \widetilde{\CC}_0 \cap B_R)$ (and similarly for $\tilde{\varphi}_t^{-1}(\surf_\varepsilon \widetilde{\CC}_t \cap B_R)$).
    
    In order to clarify this idea, we need to introduce some notation. Let $r_t$ be the nearest point retraction of $\Hyp^3$ onto the convex subset $\widetilde{\CC}_t$. Given a point $q$ of $\surf_\varepsilon \widetilde{\CC}_t$, we denote by $\altmathcal{H}_{t,q}$ the unique support half-space of $\widetilde{\CC}_t$ at $r_t(q)$ whose boundary $\partial \altmathcal{H}_{t,q} = H_{t,q}$ is orthogonal to the geodesic segment connecting $r_t(q)$ to $q$ (see Figure \ref{figure:support_planes_equidistant_surfaces}). By construction, we have the inclusion $\neigh_\varepsilon \altmathcal{H}_{t,q} \supseteq \neigh_\varepsilon \widetilde{\CC}_t$, and the surfaces $\surf_\varepsilon \altmathcal{H}_{t,q}$, $\surf_\varepsilon \widetilde{\CC}_t$ are tangent to each other at the point $q$.  In other words, given $q \in \surf_\varepsilon \widetilde{\CC}_t$, the surface $\surf_\varepsilon \altmathcal{H}_{t,q}$ lies outside $\interior(\neigh_\varepsilon \widetilde{\CC}_t)$, it approximates $\surf_\varepsilon \widetilde{\CC}_t$ at first order at $q$ and it is strictly convex, with second fundamental form described in Lemma \ref{lemma:equid_surface_plane}. Therefore, if for every $q \in \surf_\varepsilon \widetilde{C}_0 \cap B_R$ and $t \in (- t_0, t_0)$ the surface $\tilde{\varphi}_t(\surf_\varepsilon \altmathcal{H}_{0,q})$ remains convex at $\tilde{\varphi}_t(q)$, then $\tilde{\varphi}_t(\surf_\varepsilon \widetilde{\CC}_0 \cap B_R)$ has to be convex too. Analogously, the convexity of the surfaces $\tilde{\varphi}_t^{-1}(\surf_\varepsilon \altmathcal{H}_{t,q})$ at $\tilde{\varphi}^{-1}(q)$, as $q$ varies in $\surf_\varepsilon \widetilde{C}_t \cap B_R$, implies the convexity of $\varphi_t^{-1}(\surf_\varepsilon \CC M_t)$.
    
    \begin{figure}[t]
        \centering
        \includegraphics[width=0.8\textwidth]{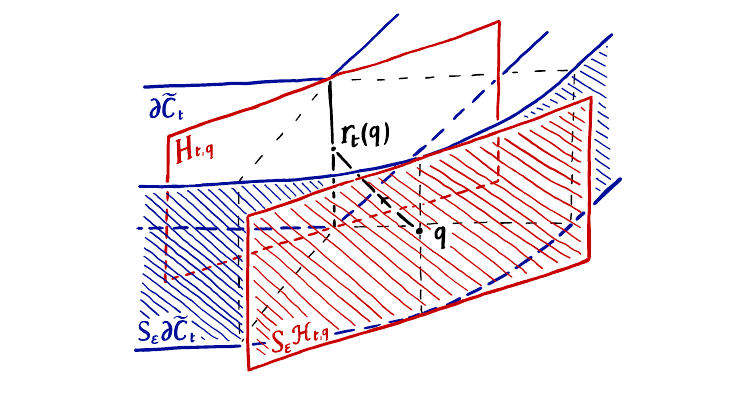}
        \caption{A schematic picture of the surface $\surf_\varepsilon \altmathcal{H}_{\varepsilon,q}$} \label{figure:support_planes_equidistant_surfaces}
    \end{figure}
    
    \vspace{0.5cm}
    
    In what follows, we state the technical result about equidistant surfaces from which Lemma \ref{lemma:key_convexity_property} will follow. Given $U$ an open set of $\Hyp^3$, we denote by $\altmathcal{S}(U,\varepsilon_0)$ the collection of those surfaces embedded in $U$ that are obtained by intersecting $U$ with an equidistant surface $\surf_\varepsilon \altmathcal{H}$, for some $\altmathcal{H}$ half-space of $\Hyp^3$ meeting $U$ and for some $0 < \varepsilon \leq \varepsilon_0$. We remark that, using the notation introduced above, for every $\varepsilon \leq \varepsilon_0$ and for every $q \in \surf_\varepsilon \widetilde{\CC}_t$, the surface $\surf_\varepsilon \altmathcal{H}_{t,q} \cap B_R$ belongs to the family $\altmathcal{S}(B_R,\varepsilon_0)$.
    
    By considering the Poincar{\'e} disk model, we can identify $\Hyp^3$ with the open unit ball $\Delta$ of $\R^3$, and functions $\mappa{f}{\Hyp^3}{\Hyp^3}$ as maps from $\Delta \subset \R^3$ to itself. If $U$ is an open set of $\R^n$, $K \subset U$ is compact and $\mappa{f}{U}{\R^m}$ is a smooth map, we define
    \begin{gather*}
        \norm{f}_{\mathscr{C}^0(K)} \defin \max_{p \in K} \norm{f(p)}_0 , \\
        \norm{f}_{\mathscr{C}^k(K)} \defin \norm{f}_{\mathscr{C}^0(K)} + \sum_{h = 1}^k \norm{D^h f}_{\mathscr{C}^0(K)}
    \end{gather*}
    for $k \geq 1$, where $\norm{\cdot}_0$ is the Euclidean (operator) norm and $D$ is the Levi-\hsk Civita connection of the Euclidean metric of $\R^n$ (if $X = \sum_i X^i e_i$ and $Y = \sum_j Y^j e_j$ are two vector fields, then $D_X Y = \sum_{i, j} X^i \partial_i Y^j e_j$). Then we have:
    
    \begin{lemma} \label{lemma:techn_convexity}
        Let $B$ be an open ball in $\Hyp^3$, let $\mappa{F}{(- t_0, t_0) \times \Hyp^3}{\Hyp^3}$ be a smooth family of diffeomorphisms $F_t = F(t,\cdot)$, satisfying $F_0 = \id_{\Hyp^3}$ and $\norm{F}_{\mathscr{C}^4((-t_0, t_0) \times \overline{B})} < \infty$, and let $\varepsilon_0$ be a positive number. Given $\Sigma \in\altmathcal{S}(B,\varepsilon_0)$, we denote by $\I^\Sigma_t$ and $\II^\Sigma_t$ the first and second fundamental forms of $F_t(\Sigma)$, respectively, as $t$ varies in $(- t_0, t_0)$. Then we can find $t_0' \in (0,t_0]$ and $D > 0$, depending only on the ball $\overline{B}$ and on $\norm{F}_{\mathscr{C}^4((-t_0, t_0) \times \overline{B})}$, such that, for every surface $\Sigma = \surf_\varepsilon \altmathcal{H} \cap B$ in $\altmathcal{S}(B,\varepsilon_0)$, we have
        \begin{equation} \label{eq:convexity}
            \II^\Sigma_t - \tanh \varepsilon \, \I^\Sigma_t \geq - D \abs{t} \, \I^\Sigma_t ,
        \end{equation}
        where we are considering the unit normal vector field on $F_t(\Sigma)$ pointing toward $F_t(\neigh_\varepsilon \altmathcal{H} \cap B)$.
    \end{lemma}
    
    \noindent Assuming momentarily this fact, we can prove Lemma \ref{lemma:key_convexity_property}:
    
    \begin{proof}[Proof of Lemma~\ref{lemma:key_convexity_property}]
        First we study the surfaces $\varphi_t(\surf_\varepsilon \CC M_0)$. Following the argument described above, we need to measure the convexity of the surfaces $\tilde{\varphi}_t(\surf_\varepsilon \altmathcal{H}_{0,q} \cap B_R)$. We apply Lemma \ref{lemma:techn_convexity} to $F_t \defin \tilde{\varphi}_t$ and $B \defin B_R$, obtaining two positive constants $t_0' \leq t_0$ and $D$, which depend only on $\norm{\tilde{\varphi}}_{\mathscr{C}^4((-t_0,t_0) \times \overline{B}_R)}$, so that the relation (\ref{eq:convexity}) holds for every $\Sigma \in \altmathcal{S}(B_R,\varepsilon_0)$. Now we choose $K_1$, $\tau_1 >0$, which will depend only on $D$ and $t_0'$, so that $\tau_1 < t_0'$, $K_1 \tau_1 \leq \varepsilon_0$ and
        \[
        \frac{\tanh K_1 \abs{t}}{2} - D \abs{t} \geq 0 \qquad \text{for every $t \in (- \tau_1, \tau_1)$} .
        \]
        We want to show that $\varphi_t(\surf_{K_1 \abs{t}} \CC M_0)$ is convex for every $t \in (- \tau_1, \tau_1)$. Let $t$ be  in $(- \tau_1,\tau_1)$ and consider $\varepsilon = K_1 \abs{t}$. By the choices we made, if $q$ is a point in $\surf_{K_1 \abs{t}} \widetilde{C}_0 \cap B_R$, then the surface $\surf_{K_1 \abs{t}} \altmathcal{H}_{0,q} \cap B_R$ belongs to $\altmathcal{S}(B_R,\varepsilon_0)$. In particular, the first and second fundamental forms $\I_t$, $\II_t$ of $\tilde{\varphi}_t(\surf_\varepsilon \altmathcal{H}_{0,q} \cap B_R)$ verify the relation (\ref{eq:convexity}) with $\varepsilon = K_1 \abs{t}$, which can be rewritten as
        \[
        \II_t - \frac{\tanh K_1 \abs{t}}{2} \, \I_t \geq \left( \frac{\tanh K_1 \abs{t}}{2} - D \abs{t} \right) \I_t .
        \]
        Because of the choices we made, the right hand side is positive semi-definite. Therefore we have
        \[
        \II_t \geq \frac{\tanh K_1 \abs{t}}{2} \, \I_t .
        \]
        In particular, the surface $\tilde{\varphi}_t(\surf_{K_1 \abs{t}} \altmathcal{H}_{0,q} \cap B_R)$ is strictly convex at the point $\tilde{\varphi}_t(q)$. Since the choice of $q \in \surf_{K_1 \abs{t}} \widetilde{C}_0 \cap B_R$ was arbitrary and the surface $\tilde{\varphi}_t(\surf_{K_1 \abs{t}} \altmathcal{H}_{0,q} \cap B_R)$ locally contains $\tilde{\varphi}_t(\surf_{K_1 \abs{t}} \widetilde{C}_0)$, the argument previously mentioned proves the convexity of $\varphi_t(\surf_{K_1 \abs{t}} \CC M_0)$ for every $t \in (- \tau_1, \tau_1)$. 
        
        Now we have to deal with the case of $\varphi_t^{-1}(\surf_\varepsilon \CC M_t)$. Fixed $t \in (- t_0,t_0)$, we define
        \begin{gather*}
            M^{(t)}_s \defin M_{t + s} , \\
            \psi^{(t)}_s \defin \varphi_{t + s} \circ \varphi_t^{-1} \vcentcolon M_0' = M_t \longrightarrow M_s' = M_{t + s}
        \end{gather*}
        for every $s \in (- s_0, s_0)$, with $s_0 = s_0(t) = t_0 - \abs{t}$. Then we apply Lemma \ref{lemma:techn_convexity} to the $1$-parameter family of diffeomorphisms $(\tilde{\psi}{}^{(t)}_s)_s$, where $\tilde{\psi}{}^{(t)}_s \defin \tilde{\varphi}_{t + s} \circ \tilde{\varphi}_t^{-1}$. By construction, the constants $s_0'$ and $D'$ only depend on $\overline{B}_{R + 1}$ and $\norm*{\tilde{\psi}{}^{(t)}}_{\mathscr{C}^4((- s_0,s_0) \times \overline{B}_{R + 1})}$. Since we can find a uniform upper bound for $\norm*{\tilde{\psi}{}^{(t)}}_{\mathscr{C}^4((- s_0,s_0) \times \overline{B}_{R + 1})}$, we can assume that $s_0'$ and $D'$ are independent of $t \in (- \tau_1, \tau_1)$. Therefore, applying the argument of the previous case to the $1$-parameter deformation $(M^{(t)}_s)_s$ and the diffeomorphisms $(\psi{}^{(t)}_s)_s$, we can select $\tau \leq s_0'$ and $K$, both independent of $t$, so that the surfaces $\psi{}_s^{(t)}(\surf_{K \abs{s}} \CC M^{(t)}_0)$ are convex for every $s \in (- \tau, \tau)$. Moreover, it is not restrictive to ask that $\tau \leq \tau_1$ and $K \geq K_1$ (this ensures that $K$ and $\tau$ work also for $\varphi_t(\surf_{K \abs{t}} \CC M_0)$). Therefore, if $t \in ( - \tau, \tau)$, then $s = - t \in (- \tau, \tau)$ and the surface
        \[
        \left. \psi{}_s^{(t)}(\surf_{K \abs{s}} \CC M^{(t)}_0) \right|_{s = - t} = \varphi_t^{-1}(\surf_{K\abs{t}} \CC M_t)
        \]
        is convex, as desired. The second part of the statement follows because of the minimality of the convex core in the family of convex subsets.
    \end{proof}
    
    It remains to prove Lemma \ref{lemma:techn_convexity}:
    
    \begin{proof}[Proof of Lemma \ref{lemma:techn_convexity}]
        Let $\alpha$ be a curve lying on some surface $\Sigma = \surf_\varepsilon \altmathcal{H} \cap B \in \altmathcal{S}(B,\varepsilon_0)$. We denote by $\alpha_t$ the curve $F_t \circ \alpha$, by $\nu_t$ the unit normal vector field of $F_t(\Sigma)$ pointing toward $F_t(\neigh_\varepsilon \altmathcal{H} \cap B)$, and by
        $\norm{\cdot}$ and $\scal{\cdot}{\cdot}$ the norm and the scalar product in the hyperbolic metric of $\Hyp^3$.
        
        Assume momentarily that we could find two universal constants $C_1$, $C_2 > 0$ (depending only on the ball $\overline{B} \subset \Hyp^3$) and a $\bar{t}_0 > 0$ (depending only on $\overline{B}$ and on the family $(F_t)_t$), such that 
        \begin{gather*}
            \abs{\norm{\alpha_t'}^2 - \norm{\alpha'}^2} \leq C_1 \norm{\alpha_t'}^2 \norm{F_t - \id}_{\mathscr{C}^1(\overline{B})} , \\
            \abs{\scal{\altmathcal{D}_{\alpha_t'} \nu_t}{\alpha_t'} - \scal{\altmathcal{D}_{\alpha'} \nu_0}{\alpha'}} = \abs{\scal{\altmathcal{D}_{\alpha_t'} \nu_t}{\alpha_t'} + \tanh \varepsilon \, \norm{\alpha'}^2} \leq C_2 \norm{\alpha_t'}^2  \norm{F_t - \id}_{\mathscr{C}^2(\overline{B})}
        \end{gather*}
        for all $t \in (- \bar{t}_0, \bar{t}_0)$ (in the last line we used the fact that $\surf_\varepsilon \altmathcal{H}$ has second fundamental form as in Lemma \ref{lemma:equid_surface_plane}).
        With such estimates, we deduce that
        \begin{align*}
            (\II^\Sigma_t - \tanh \varepsilon \, \I^\Sigma_t)(\alpha_t',\alpha_t') & = - \scal{\altmathcal{D}_{\alpha_t'} \nu_t}{\alpha_t'} - \tanh \varepsilon \, \norm{\alpha_t'}^2 \\
            & \geq \tanh \varepsilon \, \norm{\alpha'}^2 - C_2 \norm{\alpha_t'}^2  \norm{F_t - \id}_{\mathscr{C}^2(\overline{B})} - \tanh \varepsilon \, \norm{\alpha'}^2 + \\
            & \qquad \qquad \qquad \qquad \qquad  - C_1 \tanh \varepsilon \, \norm{\alpha_t'}^2 \norm{F_t - \id}_{\mathscr{C}^1(\overline{B})} \\
            & \geq - (C_1 + C_2) \, \norm{F_t - \id}_{\mathscr{C}^2(\overline{B})} \, \I_t^\Sigma(\alpha_t', \alpha_t')
        \end{align*}
        and therefore that $\II_t^\Sigma - \tanh \varepsilon \, \I^\Sigma_t \geq - (C_1 + C_2) \, \norm{F_t - \id}_{\mathscr{C}^2(\overline{B})} \, \I_t^\Sigma$ for every $t \in (- \bar{t}_0, \bar{t}_0)$. Since the map $F$ is regular in $t$, where $F_t = F(t,\cdot)$, we can find two constants $t_0'$ and $D$, depending only on $\norm{F}_{\mathscr{C}^4((- t_0, t_0) \times \overline{B})}$ and $\overline{B}$, for which the final statement holds (for this it is definitively enough to control the derivatives of order $\leq 2$ in $t$ and of order $\leq 2$ in $p \in \overline{B}$).
        
        The only thing left is to prove the two relations above. Let $g_0$ denote the Euclidean metric of $\R^3$ and $g$ the hyperbolic metric on $\Delta \cong \Hyp^3$. Identifying $\Hyp^3$ with an open set of $\R^3$, it make sense to compute a tensor $T_p$ at $p$ on vectors (or forms) lying in the tangent (or cotangent) space at a different point $q$, via the identifications $T_p \Hyp^3 \cong T_p \R^3 \cong T_q \R^3 \cong T_q \Hyp^3$. Therefore we can write:
        \begin{align*}
            \abs{\norm{\alpha_t'}^2 - \norm{\alpha'}^2} & \leq \abs{(g \circ F_t)(D_{\alpha'} F_t, D_{\alpha'} F_t) - g(\alpha',\alpha')} \\
            & \leq \abs{(g \circ F_t)(D_{\alpha'} F_t, D_{\alpha'} F_t - \alpha')} + \abs{(g \circ F_t)(D_{\alpha'} F_t - \alpha', \alpha')} + \\
            & \qquad \qquad \qquad \qquad \qquad \qquad \qquad + \abs{(g \circ F_t)(\alpha',\alpha') - g(\alpha', \alpha')} \\
            & \leq \left( \norm{g \circ F_t}_0 \norm{D_\cdot F_t}_0 \norm{D_\cdot F_t - D_\cdot \id}_0 + \norm{g \circ F_t}_0 \norm{D_\cdot F_t - D_\cdot \id}_0 + \right. \\
            & \qquad \qquad \qquad \qquad \qquad \qquad \qquad \left. + \norm{g \circ F_t - g}_0 \right) \norm{\alpha'}_0^2 ,
        \end{align*}
        where $\norm{\cdot}_0$ is the operator norm with respect to the Euclidean metric in $\R^3$. The terms $\norm{D_\cdot F_t - D_\cdot \id}_0$ and $\norm{g \circ F_t - g}_0$ can be bounded by some universal constant multiplied by $\norm{F_t - \id}_{\mathscr{C}^1(\overline{B})}$. The terms $\norm{g \circ F_t}_0$, $\norm{D_\cdot F_t}_0$ are controlled, since $F_t$ is $\mathscr{C}^1$-\hsk close to $\id$. Since $\overline{B}$ is compact and the $F_t$'s are diffeomorphisms $\mathscr{C}^1$-\hsk close to $\id$, the norms $\norm{\cdot}_0$, $\norm{D_\cdot F_t}$ and $\norm{\cdot}$ are uniformly equivalent between each other on $\overline{B}$. Combining these facts together we obtain the first inequality.
        
        For the second relation, we can proceed similarly decomposing the expression in the following way:
        \begin{align*}
            \abs{\scal{\altmathcal{D}_{\alpha_t'} \nu_t}{\alpha_t'} - \scal{\altmathcal{D}_{\alpha'} \nu_0}{\alpha'}} & \leq \abs{(g \circ F_t)( \altmathcal{D}_{\alpha_t'} \nu_t, \alpha_t' - \alpha')} + \abs{(g \circ F_t)(\altmathcal{D}_{\alpha_t' - \alpha'} \nu_t, \alpha') } + \\
            & \qquad \qquad \qquad \quad + \abs{(g \circ F_t)(\altmathcal{D}_{\alpha'} \nu_t - \altmathcal{D}_{\alpha'} \nu_0, \alpha') } + \\
            & \qquad \qquad \qquad \quad + \abs{(g \circ F_t)(\altmathcal{D}_{\alpha'} \nu_0, \alpha') - g(\altmathcal{D}_{\alpha'} \nu_0, \alpha')} \\
            & \leq 2 \norm{g \circ F_t}_0 \norm{\altmathcal{D}_\cdot \nu_t}_0 \norm{D_\cdot F_t - D_\cdot \id}_0 \norm{\alpha'}_0^2 + \\
            & \qquad \qquad \qquad \quad + \norm{g \circ F_t}_0  \norm{\altmathcal{D}_\cdot \nu_t - \altmathcal{D}_\cdot \nu_0}_0 \norm{\alpha'}_0^2 + \\
            & \qquad \qquad \qquad \quad + \norm{g \circ F_t - g}_0 \norm{\altmathcal{D}_\cdot \nu_0}_0 \norm{\alpha'}_0^2 .
        \end{align*}
        The vector field $\nu_0$ is the restriction to $\Sigma$ of the gradient $\nabla d$ of the signed distance from the plane $\partial \altmathcal{H}$ (oriented in the suitable way), \emph{independently on $\varepsilon$}. We can find two vector fields $V_1$, $V_2$ on a neighborhood of $\partial \altmathcal{H}$ so that $V_1$, $V_2$ span the tangent space of the surface $\surf_\varepsilon \altmathcal{H}$ for every $\varepsilon \leq \varepsilon_0$. The vector fields $V_1$, $V_2$ and $\nabla d$ have covariant derivatives which are uniformly bounded, as we vary $\altmathcal{H}$, since the half-spaces $\altmathcal{H}$ must meet $\overline{B}$. The vector field $\nu_t$ can be obtained as
        \[
        \frac{(F_t)_*(V_1) \times (F_t)_*(V_2)}{\norm{(F_t)_*(V_1) \times (F_t)_*(V_2)}} ,
        \]
        where $\times$ denotes the vector product. Therefore the first derivatives of $\nu_t$ are close to the ones of $\nu_0 = V_1 \times V_2 / \norm{V_1 \times V_2}$, again uniformly in the half-space $\altmathcal{H}$ meeting $\overline{B}$. This implies that the terms $\norm{\altmathcal{D}_\cdot \nu_0}_0$, $\norm{\altmathcal{D}_\cdot \nu_t}_0$ are uniformly controlled, and that $\norm{\altmathcal{D}_\cdot \nu_t - \altmathcal{D}_\cdot \nu_0}_0$ can be bounded by some universal constant multiplied by $\norm{F_t - \id}_{\mathscr{C}^2(\overline{B})}$. Combining these observations with what previously done for the first inequality, we deduce the second claimed inequality.
    \end{proof}
    
    \subsection{The variation of the dual volume}
    
    Given $\varepsilon \in [0 , \varepsilon_0]$ and $t \in (- t_0, t_0)$, we define 
    \begin{align*}
        v_\varepsilon^*(t) \defin \Vol^*_{M_t}( \neigh_\varepsilon \CC M_t), & & u_\varepsilon^*(t) \defin \Vol^*_{M_t}(\varphi_t(\neigh_\varepsilon \CC M_0)).
    \end{align*}
    Our proof of Theorem \ref{thmx:dual_bonahon_schlafli} will be divided in some steps. The function that needs to be differentiated at $t = 0$ is $V_\CC^*(M_t) = v_0^*(t)$, in the 
    notation above. However, this quantity is not easy to handle directly, because the variation of the geometric structure of $\CC M_t$ is complicated. To overcome this problem, we will first study the family of functions $u_\varepsilon^*$ in Lemma \ref{lemma:derivative_functions_u}, and the limit $\lim_\varepsilon (u_\varepsilon^*)'(0)$ in Proposition \ref{prop:dual_volume_derivative_length}. Here we will see how the differential of the length of the bending measure comes into play. Afterwards we will use the properties of the dual volume to relate $\lim_\varepsilon (u_\varepsilon^*)'(0)$ to the actual derivative $(v_0^*)'(0)$ in Proposition \ref{prop:derivative_limit_schlafli}. In this manner we will conclude that the variation of the dual volume coincides, up to multiplicative constant, with the variation of the length of the realization of the bending measure of the convex core $\mu = \mu_0$. The last part of this subsection will be dedicated to relating this result with the differential of the length function of $\mu$ over the Teichm\"uller space.
    
    \begin{lemma} \label{lemma:derivative_functions_u}
        The functions $\mappa{u^*_\varepsilon}{(- t_0, t_0)}{\R}$ are smooth in $t$, and they converge $\mathscr{C}^\infty$-\hsk uniformly to $u_0^*$ as $\varepsilon$ goes to $0$. Moreover, they satisfy
        \[
        (u^*_\varepsilon)'(0) = \frac{1}{4} \int_{\surf_\varepsilon \CC M_0} \scall{ \delta \I_\varepsilon}{\II_\varepsilon - H_\varepsilon \I_\varepsilon}_\varepsilon \dd{a}_\varepsilon ,
        \]
        where $\scall{\cdot}{\cdot}_\varepsilon$ denotes the scalar product on the space of $2$-\hsk tensors induced by $\I_\varepsilon$.
        \begin{proof}
            Let $u_\varepsilon(t)$ be $\Vol_{M_t}(\varphi_t(\neigh_\varepsilon \CC M_0))$. Then the functions $u_\varepsilon^*$ can be expressed as
            \[
            u_\varepsilon^*(t) = u_\varepsilon(t) - \frac{1}{2} \int_{\varphi_t(\surf_\varepsilon \CC M_0)} H \dd{a} .
            \]
            We prove the regularity of $u_\varepsilon^*$ in $t$ by focusing on the two terms separately. By the choice we made of the family of diffeomorphisms $(\varphi_t)_t$ at the beginning of Section \ref{sec:dual_bonahon_schlafli}, the pullback $\varphi_t^* \dvol_{M_t}$ of the volume forms of $M_t$ vary smoothly in $t$, and they can be expressed in the form $\varphi_t^* \dvol_{M_t} = f(t, \cdot) \dvol_{M_0}$, for some smooth function $\mappa{f}{(- t_0, t_0) \times M_0}{\R}$. If we denote by $A_\varepsilon$ the subset $\neigh_\varepsilon \CC M_0$ of $M_0$ for every $\varepsilon \in ]0,\varepsilon_0]$, then the functions $u_\varepsilon$ satisfy:
            \[
            u_\varepsilon(t) = \int_{M_0} \1_{A_\varepsilon} f(t, \cdot) \dvol_{M_0} ,
            \]	
            where $\1_{A_\varepsilon}$ stands for the characteristic function of the set $A_\varepsilon$ (i. e. $\1_{A_\varepsilon}(p) = 1$ if $p \in A_\varepsilon$, and $\1_{A_\varepsilon}(p) = 0$ otherwise). Observe that the sets $A_\varepsilon$ are compact and they decrease, as $\varepsilon$ goes to $0$, to $\CC M_0$. As a consequence of the regularity of $f$ in $t$, a simple application of the Lebesgue's dominated convergence theorem (see e. g. \cite{royden1988real}) proves the smoothness of the functions $u_\varepsilon$ in $t$ and their $\mathscr{C}^\infty$-\hsk uniform convergence to $u_0$. 
            
            To show the regularity of the second term of $u_\varepsilon^*$, we will describe a way to express the integral of the mean curvature as the integral of a suitable $2$-\hsk form, from which the dependence in $t$ and $\varepsilon$ will be clearer. 
            
            Consider $(M,g)$ an oriented Riemannian $3$-\hsk manifold with volume form $\dvol_M$. Given any point $(p,v)$ of the tangent bundle $T M$, the Levi-\hsk Civita connection $\nabla$ of $M$ determines a natural splitting of the tangent space $T_{(p,v)} T M$ of the form $T_{(p,v)} T M = U_{(p,v)} \oplus W_{(p,v)}$, where $U_{(p,v)}$ is the vector subspace of $T_{(p,v)} T M$ tangent to the space of $\nabla$-\hsk parallel vector fields at $p$, and $W_{(p,v)}$ is the tangent space at $(p,v)$ to the fiber $T_p M \subset T M$, which can be naturally identified with $T_p M$. The differential of the bundle map $T M \rightarrow M$ at $(p,v)$ has kernel equal to $W_{(p,v)}$, and it restricts to an isomorphism from $U_{(p,v)}$ to $T_p M$. This procedure determines a natural identification between $T_{(p,v)} T M$ and $(T_p M)^2$, which we will implicitly use in what follows. We define a $2$-\hsk form $\omega_M$ over $T^1 M$, the unit tangent bundle of $M$, as follows:
            \[
            (\omega_M)_{(p,v)}((\dot{p}, \dot{v}), (\dot{p}',\dot{v}')) \defin \scal{v}{\dot{p}'\times \dot{v} - \dot{p} \times \dot{v}'}
            \]
            where $(p,v) \in T^1 M$, $(\dot{p}, \dot{v}), (\dot{p}', \dot{v}') \in T_{(p,v)} T^1 M \subset T_{(p,v)} M$, and $\scal{\cdot}{\cdot}$ denotes the scalar product over $T_p M$. If $S$ is an embedded surface in $M$, then the choice of a normal vector field on $S$ determines a lift $\mappa{\iota}{S}{T^1 M}$, given by $\iota(p) = (p,n_p)$. Consider now $e_1, e_2$ a local orthonormal frame of $S$ diagonalizing the shape operator $B$ of $S$, i. e. $B e_i = - D_{e_i} n = \lambda_i e_i$ for $i = 1, 2$, and locally satisfying $e_1 \times e_2 = n$. Then we have:
            \begin{align*}
                (\iota^* \omega_M) (e_1, e_2) & = \omega_M((e_1, D_{e_1} n), (e_2, D_{e_2} n)) \\
                & \scal{n}{e_2 \times (- B e_1) - e_1 \times (- B e_2)} \\
                & = \scal{n}{- \lambda_1 \ e_2 \times e_1 + \lambda_2 \ e_1 \times e_2} \ \\
                & = \lambda_1 + \lambda_2 = H .
            \end{align*}
            This shows in particular that, given any surface $S \subset M$, the integral of its mean curvature can be expressed as the integral over $S$ of the $2$-\hsk form $\iota^* \omega_M$, where $\iota$ is the lift of $S$ to $T^1 M$ determined by its normal vector field. Consider now $\mappa{\psi}{M}{N}$ a diffeomorphism between two Riemannian manifolds $M$ and $N$, and define an induced map on the unit tangent bundles $\mappa{\hat{\psi}}{T^1 M}{T^1 N}$ as follows:
            \[
            \hat{\psi}(p,v) \defin \left( \psi(p), \frac{(\dd{\psi}^{-1})_{\psi(p)}^{\ad}(v)}{\norm{(\dd(\psi^{-1})_{\psi(p)}^{\ad}(v)}} \right) ,
            \]
            where $\ad$ stands for the adjoint map with respect to the scalar products on $T_p M$ and $T_{\psi(p)} N$. Given $v \in T_p^1 M$, the vector $(\dd{\psi}^{-1})_{\psi(p)}^{\ad}(v)$ is orthogonal (with respect to the metric of $N$) to the image under $\dd{\psi}_p$ of the subspace $\langle v \rangle^\perp \subset T_p M$. This property implies that, if $\mappa{\iota}{S}{M}$ is the lift of $S$ to $T^1 M$, then $\hat{\psi} \circ \iota$ parametrizes the lift of $\psi(S)$ in $T^1 N$. In particular, combining this remark with what previously observed, we see that
            \[
            \int_{\psi(S)} H \dd{a} = \int_S (\hat{\psi} \circ \iota)^* \omega_N ,
            \]
            for every embedded surface $S \subset M$ and for every diffeomorphism $\mappa{\psi}{M}{N}$ (up to sign for the choice of the normal direction). 
            
            The claimed regularity of the term in the mean curvature will now follows from this simple relation. To see this, let $E$ be the subset of $T^1 M_0$ given by the pairs $(p, \nu)$ where $p \in \partial \CC M_0$ and $\nu$ is the exterior normal direction to a support half-\hsk space of $\CC M_0$ at $p$. Observe that, if $p$ lies on an atomic leaf of the bending measured lamination with weight $\alpha$, then there is a $1$-\hsk parameter family of unit tangent vectors $(\nu_\vartheta)_{\vartheta \in [0,\alpha]}$ in $T_p^1 M_0$ satisfying $(p,\nu_\vartheta) \in E$. The subset $E$ describes a surface in the unit tangent bundle of $M_0$, which in a sense generalizes the notion of normal bundle to the singular surface $\partial \CC M_0$. If $\exp_t$ denotes the geodesic flow at time $t$ on the unit tangent bundle of $M_0$, then the lifts of the surfaces $\surf_\varepsilon \CC M_0$ in $T^1 M_0$ are parametrized by the maps $\mappa{\iota_\varepsilon}{E}{T^1 M_0}$, with $\iota_\varepsilon(p,\nu) = \exp_\varepsilon (p, \nu)$ (here the resulting normal vector field is the exterior one). The lift of a fixed surface $\surf_{\varepsilon_0} \CC M_0$ is $\mathscr{C}^{0,1}$, with Lipschitz constant of the first derivatives that a priori depends on $\varepsilon_0$. However, since the geodesic flow $(\exp_{- \varepsilon})_{\varepsilon \leq \varepsilon_0}$ is uniformly $\mathscr{C}^2$ over the compact set $T^1 M_0 |_{\neigh_{2 \varepsilon_0} \CC M_0}$, and since $\exp_{- \varepsilon'} \circ \iota_\varepsilon = \iota_{\varepsilon - \varepsilon'}$ for all $\varepsilon' < \varepsilon$, the Lipschitz constants of the first-\hsk order derivatives of the \emph{lifts} of surfaces $\surf_\varepsilon \CC M_0$ are uniformly bounded in $\varepsilon \in [0,\varepsilon_0]$ (observe that this is not the case if we look at the second-\hsk order derivatives of $\surf_\varepsilon \CC M_0$ before lifting them to the unitary tangent bundle). This remark shows in particular that the surface $E$ is $\mathscr{C}^{0,1}$, and that the functions $\iota_\varepsilon$ converge $\mathscr{C}^{0,1}$-\hsk uniformly to $\id_E$ as $\varepsilon$ goes to $0$. Let now $\omega_t = \omega_{M_t}$ denote the natural $2$-\hsk form over the manifold $T^1 M_t$ described as above. Then, by the formula we showed, we have:
            \[
            - \int_{\varphi_t(\surf_\varepsilon \CC M_0)} H \dd{a} = \int_E (\hat{\varphi}_t \circ \iota_\varepsilon)^* \omega_t = \int_E \iota_\varepsilon^* (\hat{\varphi}_t^* \omega_t) .
            \]
            Since the maps $\iota_\varepsilon$ are uniformly $\mathscr{C}^{0,1}$, the forms $(\hat{\varphi}_t \circ \iota_\varepsilon)^* \omega_t$ are $L^\infty(\Sigma, \dd{a}_E)$ uniformly in $\varepsilon$ and smooth in $t$, for fixed area form $\dd{a}_E$ on $E$ (area forms on a $\mathscr{C}^{0,1}$-\hsk surface are defined almost everywhere). In particular, by applying again the Lebesgue's dominated convergence theorem we see that the quantity $\int_{\varphi_t(\surf_\varepsilon \CC M_0)} H \dd{a}$ is smooth in $t$ and it converges $\mathscr{C}^\infty$-\hsk uniformly as $\varepsilon$ goes to $0$.
            
            Finally, the first-\hsk order variation at $t = 0$ in the statement is an immediate consequence of the differential Schl\"afli formula in Proposition \ref{prop:variation_formula_dual_volume_in_M}, and the fact that $\varphi_0 = \id$.
        \end{proof}
    \end{lemma}
    
    \begin{proposition}\label{prop:dual_volume_derivative_length}
        Assume that $(M_t)_t$ is a $1$-\hsk parameter family of convex co-compact manifolds as above. Then we have:
        \[
        (u_0^*)'(0) = \lim_{\varepsilon \rightarrow 0} (u_\varepsilon^*)'(0) = - \frac{1}{2} \iint_{\lambda} \dd{\dot{\length}} \dd{\mu} ,
        \]
        where $\mu$ is the bending measure of $\partial \CC M = \partial \CC M_0$ and $\lambda$ is a geodesic lamination containing $\supp \mu$.
        \begin{proof}
            As already observed, we can divide the surface $\surf_\varepsilon \CC M = \surf_\varepsilon \CC M_0$ in two regions:
            \begin{itemize}
                \item the open set $\surf^f_\varepsilon \defin r^{-1}(\partial \CC M \setminus \lambda) \cap \surf_\varepsilon \CC M$ ($f$ stands for flat), namely the portion of $\surf_\varepsilon \CC M$ that projects onto the union of the interior of the flat pieces of $\partial \CC M$;
                \item the closed set $\surf^b_\varepsilon \defin r^{-1}(\lambda)$ ($b$ stands for bent), namely the portion of $\surf_\varepsilon \CC M$ that projects onto the bending lamination.
            \end{itemize}
            On the portion $\surf^f_\varepsilon$ we have an explicit description of all the geometric quantities, by Lemma \ref{lemma:equid_surface_plane}. In particular, we can write the integral in terms of the hyperbolic metric on the flat parts, obtaining
            \begin{align*}
                \int_{\surf_\varepsilon^f} \scall{ \delta \I_\varepsilon}{\II_\varepsilon - H_\varepsilon \I_\varepsilon}_\varepsilon \dd{a}_\varepsilon & = \sum_{F \subset \partial \CC M_0 \setminus \lambda} \int_F (\scall{\delta \I_\varepsilon}{- \tanh \varepsilon \, \I_\varepsilon}_\varepsilon \circ r) \cosh^2 \varepsilon \dd{a}_F \\
                & = - \sinh \varepsilon \cosh \varepsilon \, \int_{\partial \CC M_0 \setminus \lambda} \scall{\delta \I_\varepsilon}{\I_\varepsilon}_\varepsilon \circ r \, \dd{a} ,
            \end{align*}
            where the sum is taken over all the flat pieces $F$ in $\partial \CC M \setminus \lambda$. The variation of the first fundamental form $\delta \I_\varepsilon$ is the restriction of $\dot{g} = \left. \dv{t} \varphi_t^* g_{M_t} \right|_{t = 0}$ to the tangent space of $\surf_\varepsilon \CC M$. In particular, since $\surf_\varepsilon \CC M$ lies in a compact set $K$ of $M = M_0$, the function $\scall{\delta \I_\varepsilon}{\I_\varepsilon}_\varepsilon$ is uniformly bounded. In conclusion, we obtain
            \[
            \lim_{\varepsilon \rightarrow 0} \int_{\surf_\varepsilon^f} \scall{ \delta \I_\varepsilon}{\II_\varepsilon - H_\varepsilon \I_\varepsilon}_\varepsilon \dd{a}_\varepsilon = - \lim_{\varepsilon \rightarrow 0} \sinh \varepsilon \cosh \varepsilon \, \int_{\partial \CC M \setminus \mu} \scall{\delta \I_\varepsilon}{\I_\varepsilon}_\varepsilon \circ r \, \dd{a} = 0 .
            \]
            Therefore, the only contribution to $\lim (u_\varepsilon^*)'(0)$ is given by $\surf^b_\varepsilon$.
            
            For convenience, we lift our study to the universal cover $\mappa{\pi}{\widetilde{M} \cong \Hyp^3}{M}$. We will first set our notation. The convex subset $\widetilde{\CC} \defin \pi^{-1}(\CC M)$ has a metric projection $\mappa{\tilde{r}}{\Hyp^3}{\widetilde{\CC}}$. Its boundary $\partial \widetilde{\CC}$ is bent along the lamination $\tilde{\lambda} \defin \pi^{-1}(\lambda)$, and it is parametrized by a locally convex pleated surface $\mappa{\tilde{f}}{\widetilde{S}}{\Hyp^3}$, having bending locus $\tilde{f}^{-1}(\tilde{\lambda})$. The preimage $\pi^{-1} (S_\varepsilon^b)$, which coincides with $\surf_\varepsilon \widetilde{\CC} \cap \tilde{r}^{-1}(\tilde{\lambda})$, will be denoted by $\widetilde{S}^b_\varepsilon$. Consider a short arc $k$ in $\widetilde{S}$ with a neighborhood $U$ on which $\tilde{f}$ is a nice embedding and set $W \defin \interior (\tilde{r}^{-1}\tilde{f}(U)) \subseteq \Hyp^3 \setminus \widetilde{\CC}$. Our actual goal is to compute
            \begin{equation} \label{eq:limit_epsilon_surface1}
                \lim_{\varepsilon \rightarrow 0} \int_{W \cap \widetilde{S}^b_\varepsilon} \scall{ \delta \I_\varepsilon}{\II_\varepsilon - H_\varepsilon \I_\varepsilon}_\varepsilon \dd{a}_\varepsilon .
            \end{equation}
            We will make use of a construction described in \cite[Section~II.2.4]{canary_epstein_green2006}: there Epstein and Marden illustrate an explicit way to extend the lamination $\tilde{\lambda}$ to a partial foliation $\altmathcal{L} = \altmathcal{L}_\eta$ of $\partial \widetilde{\CC}$, defined in the $\eta$-\hsk neighborhood (with respect its hyperbolic path metric) of $\tilde{\lambda}$, for any fixed $\eta < \log 3/2$. We briefly recall here the idea of the construction. Let $T$ be an ideal triangle in $\Hyp^2$, and denote by $U_\eta$ the $\eta$-\hsk neighborhood of $\partial T$ in $T$, with $\eta$ small. Then the region of those points in $U_\eta$ that are very close to exactly two edges of $T$, sharing an ideal vertex $v$, can be foliated using geodesic arcs asymptotic to $v$, while the region of those points that are very close to exactly one edge $e$ of $T$ can be foliated by equidistant curves from $e$. Defining a proper extension of this foliation in the regions of transition between these two behaviors in $U_\eta$, we can build a foliation on $U_\eta$ that extends the geodesic lamination of $\partial T$. Applying this construction to each ideal triangle in the pleated boundary of $\widetilde{C}$, we can construct the desired extension $\altmathcal{L}$ (see \cite[Section~II.2.4]{canary_epstein_green2006} for a more precise description).
            
            Up to taking a smaller neighborhood $U$ of $k$, we can assume that $\tilde{f}(U) \subset \bigcup \altmathcal{L}$ and we can choose a continuous orientation of the foliation $\altmathcal{L} \cap \tilde{f}(U)$. Analogously to what is done in \cite[Section~II.2.11]{canary_epstein_green2006}, we define three orthonormal vector fields on $W$ as follows:
            \begin{enumerate}
                \item the first vector field $\nu$ is given by the opposite of the gradient of the distance from $\widetilde{\CC}$;
                \item the second vector field $E_1$ is defined in terms of the oriented foliation $\altmathcal{L} \cap \tilde{f}(U)$. If $p$ lies in $W$, its projection $r(p)$ belongs to an oriented leaf $\tilde{f}(\gamma)$ of $\altmathcal{L} \cap \tilde{f}(U)$. We denote by $w$ the unitary vector of $T_{r(p)} \Hyp^3$ tangent to $\tilde{f}(\gamma)$, and we define $E_1(p)$ to be the parallel translation of $w$ along the geodesic arc in $\Hyp^3$ connecting $r(p)$ to $p$.
                \item the last vector field $E_2$ is defined requiring that $(E_1,E_2,\nu)$ is a positively oriented orthonormal frame of $T \Hyp^3$ in $W$ (assume we have fixed an orientation of $\Hyp^3$ since the beginning).
            \end{enumerate}
            Observe that the $E_i$'s are tangent to the surfaces $\surf_\varepsilon \widetilde{\CC} \cap W$, since they are orthogonal to the gradient of the distance. Therefore, they define two orthogonal oriented foliations on $\surf_\varepsilon \widetilde{\CC} \cap W$ for every $\varepsilon$. Moreover, if $r(p) \in \tilde{\lambda}$, then $E_1(p)$ is a principal direction for the equidistant surface $\surf_\varepsilon \widetilde{\CC}$ passing through $p$. In particular, we have that $\II_\varepsilon(E_1,E_1) \equiv \tanh \varepsilon$ (it is a direct consequence of the relations in Lemma \ref{lemma:equid_surface_line}). Expanding the expression $\scall{\delta \I_\varepsilon}{\II_\varepsilon - H_\varepsilon \I_\varepsilon}_\varepsilon$ in terms of this orthonormal frame over $W \cap \widetilde{S}^b_\varepsilon$ we have
            \begin{align*}
                \scall{ \delta \I_\varepsilon}{\II_\varepsilon - H_\varepsilon \I_\varepsilon}_\varepsilon & = - (\delta \I_\varepsilon)(E_1,E_1) \II_\varepsilon(E_2,E_2) - (\delta \I_\varepsilon)(E_2,E_2) \II_\varepsilon(E_1,E_1) \\
                & = - (\delta \I_\varepsilon)(E_1,E_1) \II_\varepsilon(E_2,E_2) + O(\dot{g}|_K;\varepsilon) .
            \end{align*}
            Since the area of $W \cap \widetilde{S}^b_\varepsilon$ goes to $0$ as $\varepsilon$ goes to $0$, the integral of the term $O(\dot{g}|_K;\varepsilon)$ in the expression (\ref{eq:limit_epsilon_surface1}) has limit $0$. In the end, it remains to study
            \[
            \lim_{\varepsilon \rightarrow 0} \int_{W \cap \widetilde{S}^b_\varepsilon} (\delta \I_\varepsilon)(E_1,E_1) \II_\varepsilon(E_2,E_2) \dd{a}_\varepsilon = \lim_{\varepsilon \rightarrow 0} \int_{W \cap \widetilde{S}^b_\varepsilon} (\delta \I_\varepsilon)_{1 1} (\II_\varepsilon)_{2 2} \dd{a}_\varepsilon .
            \]
            
            We denote by $\altmathcal{L}^1_\varepsilon$, $\altmathcal{L}^2_\varepsilon$ the foliations on $\widetilde{S}^b_\varepsilon \cap W$ tangent to $E_1$, $E_2$, and by $\dd{\length^1_\varepsilon}$, $\dd{\length^2_\varepsilon}$ their length elements, respectively. Then we can write
            \begin{equation} \label{eq:limit_epsilon_surface2}
                \int_{W \cap \widetilde{S}^b_\varepsilon} (\delta \I_\varepsilon)_{1 1} (\II_\varepsilon)_{2 2} \dd{a}_\varepsilon = \int_{\altmathcal{L}^2_\varepsilon} \left( \int_{\altmathcal{L}^1_\varepsilon} (\delta \I_\varepsilon)_{1 1} \dd{\length^1_\varepsilon} \right) (\II_\varepsilon)_{2 2} \dd{\length^2_\varepsilon} .
            \end{equation}
            Now it is time to see how this expression behaves in the finitely bent case. Assume that $\tilde{f}(U)$ meets a unique geodesic arc $\gamma$ in $\tilde{\lambda}$ with bending angle $\theta_0$. Then, in the coordinates described in Lemma \ref{lemma:equid_surface_line}, the vector fields $E_1$ and $E_2$ can be written as $E_1 = (\cosh \varepsilon)^{-1} \partial^\varepsilon_s$, $E_2 = (\sinh \varepsilon)^{-1} \partial^\varepsilon_\theta$. Therefore the following relations hold
            \begin{align*}
                (\delta \I_\varepsilon)_{1 1} \dd{\length^1_\varepsilon} = \frac{\dot{g}(\partial^\varepsilon_s,\partial^\varepsilon_s)}{\cosh^2 \varepsilon} \dd(\cosh \varepsilon \, s) & & (\II_\varepsilon)_{2 2} \dd{\length^2_\varepsilon} = \frac{\cosh \varepsilon}{\sinh \varepsilon} \, \dd(\sinh \varepsilon \, \theta) .
            \end{align*}
            where $\dot{g} = \dot{g}_0$. In particular, the limit as $\varepsilon \rightarrow 0$ of the expression (\ref{eq:limit_epsilon_surface2}) becomes
            \[
            \lim_{\varepsilon \rightarrow 0} \int_{\altmathcal{L}^2_\varepsilon} \left( \int_{\altmathcal{L}^1_\varepsilon} (\delta \I_\varepsilon)_{1 1} \dd{\length^1_\varepsilon} \right) (\II_\varepsilon)_{2 2} \dd{\length^2_\varepsilon} = \theta_0 \int_\gamma \dot{g}(\gamma', \gamma' ) \dd{\length} = 2 \iint_{\tilde{\lambda} \cap W} \dd{\dot{\length}} \dd{\mu} .
            \]
            To prove this relation in the general case, we make use of the standard approximations of Definition \ref{def:standard_approx}. The bending measures along the arc $k$ of the finitely bent approximations $\tilde{f}_n$ weak*-converge to $\mu$ along $k$; the $\varepsilon$-\hsk surfaces from the $\tilde{f}_n$'s converge $\mathscr{C}^{1,1}$-\hsk uniformly to $W \cap \surf_\varepsilon \widetilde{C}$; the vector fields $E_{1,n}$, $E_{2,n}$ and $\nu_n$, defined from the surface $\tilde{f}_n(U)$, converge uniformly to $E_1$, $E_2$ and $\nu$ over all the compact subsets of $W$. From these properties, the relation we proved in the finitely bent case extends to the general one.
            
            Finally, a suitable choice of a partition of unity on a neighborhood of the bending lamination $\mu$, combined with Lemma \ref{lemma:derivative_functions_u}, proves the statement.
        \end{proof}
    \end{proposition}
    
    \begin{proposition} \label{prop:derivative_limit_schlafli}
        Assume $(M_t)_t$ is a $1$-\hsk parameter family of convex co-compact manifolds as above. Then there exists the derivative of $V_\CC^*(M_t)$ at $t = 0$ and it verifies
        \[
        \dd{V_\CC^*}(\dot{M}) = - \frac{1}{2} \iint_{\lambda} \dd{\dot{\length}} \dd{\mu}.
        \]
        \begin{proof}
            The left-hand side is nothing but the limit of the incremental ratio of the function $v_0^*$ at $t = 0$. Let $K$, $\tau$ be the constants furnished by Lemma \ref{lemma:key_convexity_property}. We split our incremental ratio as follows:
            \begin{align*}
                \frac{v^*_0(t) - v^*_0(0)}{t} & = \underbrace{\frac{u^*_{K \abs{t}}(t) - u^*_{K \abs{t}}(0)}{t}}_{\text{term 1}} + \underbrace{\frac{v^*_{K \abs{t}}(0) - v^*_0(0)}{t}}_{\text{term 2}} - \underbrace{\frac{u^*_{K \abs{t}}(t) - v^*_0(t)}{t}}_{\text{term 3}} ,
            \end{align*}
            where we used the fact that $u^*_\varepsilon(0) = v^*_\varepsilon(0)$ for all $\varepsilon > 0$. In Lemma \ref{lemma:derivative_functions_u}, we showed that the functions $u^*_\varepsilon$ are smooth in $t$ and that they converge $\mathscr{C}^\infty$-\hsk uniformly to $u_0^*$ as $\varepsilon$ goes to $0$. Using the first-\hsk order expansion of $u_\varepsilon^*$ at $t = 0$ and evaluating for $\varepsilon = K \abs{t}$, we have:
            \[
            \frac{u^*_{K \abs{t}}(t) - u^*_{K \abs{t}}(0)}{t} = (u^*_{K \abs{t}})'(0) + O((u_{K \abs{t}})''(\xi_t); t) ,
            \]
            where the constant involved in the $O(t)$ depends a priori on the value of $(u_{K \abs{t}}^*)''$ in a point $\xi_t$ close to $0$. However, thanks to the $\mathscr{C}^\infty$-\hsk uniform convergence of the functions $(u_\varepsilon^*)_\varepsilon$, the second derivatives $(u_\varepsilon^*)''$ can be bounded uniformly in $\varepsilon$ over a small neighborhood of $0$, so the term $O((u_{K \abs{t}})''(\xi_t); t)$ is an actual $O(t)$. By Proposition \ref{prop:dual_volume_derivative_length} we conclude that the limit of the first term in the decomposition above is equal to $- \frac{1}{2} \iint_{\lambda} \dd{\dot{\length}} \dd{\mu}$. In what follows we will show that the second and third terms of the splitting of the incremental ratio are converging to $0$ as $t$ goes to $0$.
            
            By Proposition \ref{prop:dual_volume_order_2} applied to the $3$-\hsk manifold $M_s$, for every $\varepsilon > 0$ we have
            \begin{equation} \label{eq:applying_lemma_order_2}
                v^*_\varepsilon(s) - v^*_0(s) = \Vol^*_{M_s}(\neigh_\varepsilon \CC M_s) - \Vol^*_{M_s}(\CC M_s) = O(\length_{m_s}(\mu_s),\chi(\partial \CC M_s); \varepsilon^2) .
            \end{equation}
            In particular, for $s = 0$ and $\varepsilon = K \abs{t}$, this relation proves that the second term goes to $0$.
            
            Let $L > 1$ be a constant so that all the diffeomorphisms $\varphi_t$ are $L$-\hsk Lipschitz on a large compact set in $M_0$ containing the convex core $\CC M_0$. It is immediate to see that the following properties hold:
            \begin{align*}
                \varphi_t(\neigh_\varepsilon \CC M_0) & \subseteq \neigh_{L \varepsilon} \varphi_t(\CC M_0) \qquad \text{for every $\varepsilon > 0$} , \\
                \neigh_{\varepsilon'} \neigh_\varepsilon \CC M_t & \subseteq \neigh_{\varepsilon' + \varepsilon} \CC M_t \qquad \text{for every $\varepsilon'$, $\varepsilon > 0$} .
            \end{align*}
            Applying Lemma \ref{lemma:key_convexity_property} to the $3$-\hsk manifold $M_t$ and using the inclusion relations above, we obtain the following chain:
            \[
            \CC M_t \subseteq \varphi_t(\neigh_{K \abs{t}} \CC M_0) \subseteq \neigh_{L K \abs{t}} \varphi_t(\CC M_0) \subseteq \neigh_{L K \abs{t}} \neigh_{K t} \CC M_t \subseteq \neigh_{(L + 1) K \abs{t}} \CC M_t .
            \]
            for all $t \in (- \tau, \tau)$. All the submanifolds involved are compact convex subsets of $M_t$, hence we are allowed to consider their dual volumes. Using the monotonicity of $\Vol^*_{M_t}$, proved in Proposition \ref{prop:Vol^*_monotonic}, we get
            \[
            v_0^*(t) \geq u^*_{K \abs{t}}(t) \geq v^*_{(L + 1) K \abs{t}}(t) \qquad \text{for all $t \in (- \tau, \tau)$} .
            \]
            Applying this to estimate the third term, we obtain
            \begin{equation} \label{eq:order_2_relation}
                0 \geq \frac{u^*_{K\abs{t}}(t) - v^*_0(t)}{t} \geq \frac{v^*_{(L + 1)K\abs{t}}(t) - v^*_0(t)}{t} .
            \end{equation}
            Since the constants $K$ and $L$ only depend on the family $(\varphi_t)_t$, if we apply the equation (\ref{eq:applying_lemma_order_2}) with $s = t$ and $\varepsilon = (L + 1)K\abs{t}$, we get
            \[
            v^*_{(L + 1)K\abs{t}}(t) - v^*_0(t) = O((\varphi_t)_t,\length_{m_t}(\mu_t), \chi(\partial \CC M_t);t^2) .
            \]
            Consequently, the right side in the inequality (\ref{eq:order_2_relation}) goes to $0$ as $t$ goes to $0$, and so does the third term, which concludes the proof.
        \end{proof}
    \end{proposition}
    
    Given $\mu \in \MesLam(S)$, we define the \emph{length function of $\mu$} as the map $\mappa{L_\mu}{\Teich(S)}{\R_{\geq 0}}$ from the Teichm\"uller space of $S$ to $\R_{\geq 0}$ which associates to the hyperbolic metric $m \in \Teich(S)$ the length of $\mu$ with respect to the metric $m$. The functions $L_\mu$ are real-analytic, since they are restrictions of holomorphic functions over the set of quasi-Fuchsian groups (see \cite[Corollary~2.2]{kerckhoff1985earthquakes}).
    
    The dependence of the geometry of the convex core $\CC M$ on the hyperbolic structure of $M$ is a subtle problem. In \cite{keen_series1995continuity} the authors established the continuity of the hyperbolic metric and the bending measure of $\partial \CC M$ with respect to the structure of $M$. A much more sophisticated analysis, involving the notion of \emph{H{\"o}lder cocycles}, allowed Bonahon to describe more precisely the regularity of these maps, as done in \cite{bonahon1998variations}. In the following, we recall a parametrization result from \cite{bonahon1996shearing}, which was an essential tool in the study of \cite{bonahon1998variations}. 
    
    Fixed a maximal lamination $\lambda$ on a surface $S$, we say that a representation $\rho$ of $\pi_1(S)$ in $\Iso^+(\Hyp^3)$ realizes $\lambda$ if there exists a pleated surface $\tilde{f}$ with holonomy $\rho$ and pleating locus contained in $\lambda$. Let $\altmathcal{R}(\lambda)$ be the set of conjugacy classes of homomorphisms realizing $\lambda$, which is open in the character variety of $\pi_1(S)$ and in bijection with the space of pleated surfaces with bending locus $\lambda$, up to a natural equivalence relation. \cite[Theorem~31]{bonahon1996shearing} describes a biholomorphic parametrization of $\altmathcal{R}(\lambda)$ in terms of the hyperbolic metric and the bending cocycle of the pleated surface realizing $\rho \in \altmathcal{R}(\lambda)$. In particular, we denote by $\mappa{\psi_\lambda}{\altmathcal{R}(\lambda)}{\Teich(S)}$ the map associating to $[\rho]$ the hyperbolic metric of the pleated surface with holonomy $\rho$. 
    
    Now, let $M$ be a hyperbolic convex co-compact manifold. Denote by $\altmathcal{QD}(M)$ the space of quasi-isometric deformations of $M$, and by $\altmathcal{R}(\partial \CC M)$ the representation variety of $\pi_1(\partial \CC M)$ in $\Iso^+(\Hyp^3)$. We have a natural map $\mappa{R}{\altmathcal{QD}(M)}{\altmathcal{R}(\partial \CC M)}$ which associates to a convex co-compact hyperbolic structure $M'$ on $M$ the conjugacy class of the holonomy $[\rho']$ of $\partial \CC M'$. If $\lambda$ is a maximal lamination of $\partial \CC M'$ extending the support of the bending measure of $\partial \CC M'$, then $\psi_\lambda$ is defined on a open neighborhood of $[\rho']$, therefore we are allowed to consider the map $\psi_\lambda \circ R$. The result of \cite{bonahon1998variations} we need is the following:
    
    \begin{theorem}[{\cite[Theorem~1]{bonahon1998variations}}] \label{thm:metric_core_C1}
        Let $M$ be a hyperbolic convex co-compact manifold and denote by $\altmathcal{QD}(M)$ the space of quasi-isometric deformations of $M$. Then the map $\mappa{Q}{\altmathcal{QD}(M)}{\Teich(\partial \CC M)}$ associating to the structure $M'$ the hyperbolic metric on $\partial \CC M'$, is continuously differentiable. Moreover, given any maximal lamination extending the support of the bending measure of $\CC M'$, the differential of $Q$ at $M'$ coincides with the differential of the map $\psi_\lambda \circ R$ at $M'$.
    \end{theorem}
    
    We are finally ready to prove the variation formula for the dual volume of the convex core of a convex co-compact hyperbolic manifold:
    
    \addtocounter{theoremx}{-1}
    
    \begin{theoremx}
        Let $(M_t)_t$ be a smooth $1$-\hsk parameter family of quasi-isometric hyperbolic convex co-compact manifolds, with $M_0 = M$. Denote by $\mu \in \MesLam(\partial \CC M)$ the bending measure of the convex core of $M$ and let $t \mapsto m_t \in \Teich(\partial \CC M)$ be the family of hyperbolic metrics $m_t$ associated to the boundary of the convex core $\CC M_t$ at the time $t$. Then the dual volume of the convex core $V_\CC^*(M_t)$ admits derivative at $t = 0$, and it verifies
        \[
        \dd{V_\CC^*} (\dot{M}) = - \frac{1}{2} \dd{L_{\mu}}(\dot{m}) .
        \]
        \begin{proof}
            By Proposition \ref{prop:derivative_limit_schlafli}, the derivative of $V_\CC^*(M_t)$ at $t = 0$ exists and it coincides with $\lim_\varepsilon (u^*_\varepsilon)'(0)$. By Proposition \ref{prop:dual_volume_derivative_length}, we have the equality
            \[
            \dd{V_\CC^*} (\dot{M}) = - \frac{1}{2} \iint_\lambda \dd{\dot{\length}} \dd{\mu} ,
            \]
            where $\lambda = \supp \mu$. By Theorem \ref{thm:metric_core_C1}, given a maximal lamination $\lambda$ containing $\lambda = \supp \mu$, the variation of the hyperbolic metric $\tilde{m}_t$ of the pleated surface in $M_t$ realizing $\lambda$ coincides with the variation of the hyperbolic metric $m_t$ on the boundary of the convex core $\CC M_t$. By definition, the quantity $\iint \dd{\dot{\length}} \dd{\mu}$ coincides with $\dv{t} L_\mu (\tilde{m}_t) |_{t = 0}$. Therefore, we obtain that
            \[
            \dd{L_\mu}(\dot{m}) = \iint_\lambda \dd{\dot{\length}} \dd{\mu},
            \]
            which proves the statement.
        \end{proof}
    \end{theoremx}


\emergencystretch=1em

\printbibliography[heading=bibintoc,title={References}]

\end{document}